\documentclass[a4paper,12pt]{article}

\usepackage[utf8]{inputenc}
\usepackage[english]{babel}
\usepackage{mleftright}
\usepackage{enumitem}
\usepackage{fullpage}
\usepackage{graphicx}
\usepackage{microtype}
\usepackage{amsmath}
\usepackage{amssymb}
\usepackage{amsthm}
\usepackage{cite}
\usepackage[hidelinks]{hyperref}

\theoremstyle{definition}
\newtheorem{dfn}{Definition}[section]
\newtheorem{asm}[dfn]{Assumption}
\newtheorem{exm}[dfn]{Example}
\newtheorem{rem}[dfn]{Remark}

\theoremstyle{plain}
\newtheorem{lem}[dfn]{Lemma}
\newtheorem{thm}[dfn]{Theorem}

\newtheorem{alg}[dfn]{Algorithm}

\newcommand{\BddMul}{w}
\newcommand{\R}{\mathbb{R}}
\newcommand{\N}{\mathbb{N}}
\newcommand{\Lag}{\mathcal{L}}
\newcommand{\wto}{\rightharpoonup}
\newcommand{\dist}[1]{\operatorname{dist}(#1)}
\newcommand{\dual}[1]{\mleft\langle #1 \mright\rangle}
\newcommand{\scal}[1]{\mleft( #1 \mright)}

\newcommand{\TngCone}[2]{\mathcal{T}_{#1}(#2)}
\newcommand{\NorCone}[2]{\mathcal{N}_{#1}(#2)}

\title{
    On Error Bounds and Multiplier Methods for Variational Problems in Banach Spaces%
    \thanks{This research was supported by the German Research Foundation (DFG) within the priority
    program ``Non-smooth and Complementarity-based Distributed Parameter Systems: Simulation and
    Hierarchical Optimization'' (SPP 1962) under grant number KA 1296/24-1.}
}

\date{March 27, 2018}

\author{
    Christian Kanzow$^{\dagger}$ \and Daniel Steck%
    \thanks{University of W\"urzburg, Institute of Mathematics, Campus Hubland Nord,
        Emil-Fischer-Str.\ 30, 97074 Würzburg, Germany;
        \{kanzow,daniel.steck\}@mathematik.uni-wuerzburg.de}
}

\begin{document}

\maketitle

{
\small\textbf{\abstractname.}
This paper deals with a general form of variational problems in Banach spaces which encompasses variational inequalities as well as minimization problems. We prove a characterization of local error bounds for the distance to the (primal-dual) solution set and give a sufficient condition for such an error bound to hold. In the second part of the paper, we consider an algorithm of augmented Lagrangian type for the solution of such variational problems. We give some global convergence properties of the method and then use the error bound theory to provide estimates for the rate of convergence and to deduce boundedness of the sequence of penalty parameters. Finally, numerical results for optimal control, Nash equilibrium problems, and elliptic parameter estimation problems are presented.
\par\addvspace{\baselineskip}
}

{
\small\textbf{Keywords.}
Variational problem, variational inequality, error bound, augmented Lagrangian method, local convergence, global convergence, Nash equilibrium problem.
\par\addvspace{\baselineskip}
}

{
\small\textbf{AMS subject classifications.}
49K, 49M, 65K, 90C.
\par\addvspace{\baselineskip}
}

\section{Introduction}\label{Sec:Intro}

This paper deals with the following variational problem:
\begin{equation}\label{Eq:VI}
    \text{Find }x\in M\text{ such that}\quad\dual{F(x),v}\ge 0
    \quad\forall v\in\TngCone{M}{x},
\end{equation}
where $M\subseteq X$ is a nonempty closed set, $X$ a real Banach space, and $F:X\to X^*$ a given mapping. The set $\TngCone{M}{x}$ denotes the (Bouligand) tangent cone \cite{Bonnans2000} to $M$ at $x$. If $M$ is additionally convex, then \eqref{Eq:VI} is equivalent to
\begin{equation}\label{Eq:VI_Convex}
    \text{Find }x\in M\text{ such that}\quad\dual{F(x),y-x}\ge 0
    \quad\forall y\in M,
\end{equation}
which is often regarded as the standard form of a variational inequality (VI). Throughout this paper, we will use the terms ``variational inequality'' and ``variational problem'' interchangeably, and often refer to \eqref{Eq:VI} as a VI. Note that, in the absence of convexity, \eqref{Eq:VI} is the canonical formulation of variational problems; in particular, this form encompasses first-order necessary conditions for nonlinear optimization problems of the type
\begin{equation}\label{Eq:Opt}
    \min\ f(x) \quad\text{s.t.}\quad x\in M
\end{equation}
by choosing $F:=f'$. Throughout this paper, we assume that $M$ is given in the form
\begin{equation}\label{Eq:M}
    M=\{ x\in X: g(x)\in K \},
\end{equation}
where $g:X\to H$ is a given mapping, $H$ a real Hilbert space, and $K\subseteq H$ a nonempty closed convex set (not necessarily a cone). We make no blanket convexity assumptions on $g$ (although some of our results do pertain to the convex case). Hence, the set $M$ is nonconvex in general, and \eqref{Eq:VI} is the natural framework for our setting.

Variational inequalities are a well-known and popular class in both finite and infinite-dimensional optimization since they unify various problem types such as constrained minimization and equilibrium-type problems, in particular Nash and (certain) generalized Nash equilibrium problems \cite{Facchinei2007,Facchinei2010,Fischer2014,Hintermueller2015,Kanzow2017a}. This opens up a broad spectrum of applications including optimal control, parameter estimation, differential games, and problems in mechanics or shape optimization. Many further applications are given in \cite{Baiocchi1984,Glowinski2015,Glowinski1981,Kinderlehrer2000}. As a result, VIs have gained considerable attention in the literature and a variety of algorithms have been developed for their solution, e.g.\ \cite{Facchinei2003,Fortin1983,Glowinski2008,Ulbrich2011}.

On the other hand, the augmented Lagrangian method (ALM, also called multiplier-penalty method or simply multiplier method) is one of the classical methods for nonlinear optimization, see \cite{Conn1991,Hestenes1969,Powell1969,Rockafellar1973,Rockafellar1974} and the textbooks \cite{Bertsekas1982,Nocedal2006}. In recent years, ALMs have seen a certain resurgence \cite{Andreani2007,Andreani2008,Birgin2012,Birgin2010,Birgin2014} in the form of modified methods which use a slightly different update of the Lagrange multiplier and turn out to have very strong global convergence properties \cite{Birgin2014}. A comparison of the classical and modified ALMs is given in \cite{Kanzow2017}. We also note that ALMs have been generalized to VIs in finite dimensions \cite{Andreani2008} and to infinite-dimensional optimization problems in certain restricted settings \cite{Hintermueller2006,Ito1990a,Ito1990b,Ito2000,Ito2008,Kanzow2016,Wierzbicki1977}. However, most of these papers either consider rather specific problem settings \cite{Hintermueller2006,Ito1990a,Ito1990b,Ito2000,Ito2008} or deal with global convergence properties only \cite{Kanzow2016}.

The main purpose of the present paper is to analyze the local convergence properties of ALMs for variational inequalities in the general (possibly infinite-dimensional) setting \eqref{Eq:VI}. To accomplish this, we will need certain elements of perturbation and error bound theory for generalized equations and KKT systems, some of which are refinements of the corresponding results in finite dimensions \cite{Ding2017,Dontchev1998,Fischer2002,Izmailov2012a}. Using these, we will prove that, given a KKT point which admits a primal-dual error bound, the ALM converges locally to this point with a rate of convergence that is essentially $1/\rho_k$ (where $\rho_k$ is the penalty parameter), and that $\{\rho_k\}$ remains bounded if updated suitably.

Sufficient conditions for the primal-dual error bound include a suitable second-order sufficient condition (SOSC) together with a strict version of the Robinson constraint qualification (see Section \ref{Sec:Prelims}). These assumptions are akin to those used in \cite{Birgin2012} for ALMs in finite-dimensional nonlinear programming (NLP), where the authors obtain results similar to ours. Interestingly, however, it turns out that these results (for standard NLP) can be established under SOSC only \cite{Fernandez2012} by using the specific structure of the constraints. In particular, when transferred to our notation, the set $K$ arising from NLP is polyhedral and this yields, roughly speaking, the dual part of the error bound without any constraint qualification \cite{Fernandez2012,Izmailov2012a}. However, apart from the NLP setting, polyhedrality is a rare property which is usually violated, e.g.\ in optimal control or semidefinite programming. As a result, SOSC alone does not yield a primal-dual error bound, see the example in Section \ref{Sec:ErrorBounds}. We solve this issue by using SOSC together with a suitable constraint qualification.

The paper is organized as follows. We start with some preliminary material in Section \ref{Sec:Prelims} and give some results on primal-dual error bounds in Section \ref{Sec:ErrorBounds}. Section \ref{Sec:Method} contains a precise statement of our algorithm and we continue with some global convergence results in Section \ref{Sec:GlobalConv}. In Section \ref{Sec:LocalConv}, we prove the main results of this paper, i.e.\ local convergence of the ALM under the error bound hypothesis. We then give some numerical results in Section \ref{Sec:Applic} and final remarks in Section \ref{Sec:Final}.

\textbf{Notation:} Throughout the paper, $X$ is always a real Banach space, $H$ a real Hilbert space, and their duals are denoted by $X^*$ and $H^*$, the latter of which we usually identify with $H$. Fr\'echet-derivatives are denoted by a prime $'$ or by $D_x$ if the variable is emphasized, and we use the abbreviation lsc for lower semicontinuity. Strong and weak convergence are denoted by $\to$ and $\wto$, respectively. Duality pairings are written as $\dual{\cdot,\cdot}$, scalar products as $\scal{\cdot,\cdot}$, and norms are denoted by $\|\cdot\|$ with an appropriate subscript to emphasize the corresponding space (e.g.\ $\|\cdot\|_X$). If $S$ is a nonempty subset of some normed space, we write $d_S=\dist{\cdot,S}$ for the distance to $S$. Additionally, if $S\subseteq H$ is closed and convex, we write $P_S$ for the projection onto $S$.

\section{Preliminaries}\label{Sec:Prelims}

This section is dedicated to establishing some preliminary results as well as fixing the setting we will consider later. Recall that the set $M$ is given by the formula \eqref{Eq:M} with a nonempty closed convex set $K\subseteq H$.

\subsection{Cones and Convexity}

If $S$ is a nonempty closed subset of some space $Z$, then $S^{\circ}:=\{ \psi\in Z^*: \dual{\psi,s}\le 0~\forall s\in S \}$ denotes the \emph{polar cone} of $S$. If $Z$ is a Hilbert space, we of course treat $S^{\circ}$ as a subset of $Z$. Moreover, if $x\in S$ is a given point, we denote by
\begin{equation*}
    \TngCone{S}{x}:= \bigl\{ d\in Z : \exists x^k\to x,\, t_k\downarrow 0
    \text{ such that } x^k\in S \text{ and } (x^k-x)/t_k \to d \bigr\}
\end{equation*}
the \emph{tangent cone} of $S$ at $x$. If $S$ is additionally convex, we also define the \emph{normal cone}
\begin{equation*}
    \NorCone{S}{x}:= \mleft\{ \psi\in Z^*: \dual{\psi,y-x}\le 0~\forall y\in S \mright\}
    =\TngCone{S}{x}^{\circ}.
\end{equation*}
If $x\notin S$, we define $\TngCone{S}{x}$ and $\NorCone{S}{x}$ to be empty. Note that, if $S$ is a convex set, then $\TngCone{S}{x}$ and $\NorCone{S}{x}$ are closed convex cones for all $x\in S$.

Recall that the constraint system of the VI is given by $g(x)\in K$ with $K\subseteq H$ a nonempty closed convex set. A natural question is what the appropriate notion of convexity is in this general setting. In particular, we would like to give sufficient conditions for the convexity of the feasible set $M$. To this end, consider the recession cone
\begin{equation}\label{Eq:RecessionCone}
    K_{\infty}:=\{ y\in H: y+K\subseteq K \}.
\end{equation}
It is well-known that $K_{\infty}$ is a nonempty closed convex cone \cite{Bauschke2011,Bonnans2000}. If $K$ itself is a cone, then $K_{\infty}=K$. We associate with $K$ (and $K_{\infty}$) the (partial) order relation
\begin{equation}\label{Eq:OrderK}
    y \le_K z :\Longleftrightarrow z-y\in K_{\infty}.
\end{equation}
Note that we use the notation $\le_K$ for the sake of convenience, even though the order is actually induced by the cone $K_{\infty}$. We also note that $K_{\infty}$ may not be pointed (that is, $K_{\infty}\cap (-K_{\infty})$ may contain a nonzero element) and, hence, the relation $\le_K$ does not necessarily satisfy the antisymmetry property
\begin{equation*}
    a\le_K b \wedge b\le_K a \implies a=b.
\end{equation*}
In the terminology of order theory, this makes $\le_K$ a so-called preorder. We will simply call it an order relation due to the descriptiveness of the term. Note also that, throughout this paper, the symbol $\le$ without any index is always the standard ordering in $\R$.

The order relation \eqref{Eq:OrderK} allows us to extend various familiar concepts from finite-dimensional optimization to our setting. For instance, we say that $g$ is convex if
\begin{equation*}
    g(\alpha x+(1-\alpha)y) \le_K \alpha g(x) + (1-\alpha) g(y)
\end{equation*}
holds for all $x,y\in X$ and $\alpha\in[0,1]$. Other notions which involve an order such as increasing, decreasing or concave functions are also defined in a straightforward way. For example, the distance function $d_K:H\to\R$ is decreasing since $z\ge_K y$ implies $z=y+k$, $k\in K_{\infty}$, and
\begin{equation*}
    d_K(z)=d_K(y+k)\le \|y+k-(P_K(y)+k)\| = \|y-P_K(y)\| = d_K(y),
\end{equation*}
where the inequality uses the fact that $P_K(y)+k\in K$ by definition of $K_{\infty}$. Some other results pertaining to convexity, concavity, etc.\ are given in the following lemma. Note that, in the context of our constraint set \eqref{Eq:M} with $g(x)\in K$, it is more natural to consider concavity of $g$ with respect to the ordering \eqref{Eq:OrderK} as opposed to convexity.

\begin{lem}\label{Lem:GeneralizedConvexity}
    Assume that $g:X\to H$ is concave. If $m:H\to\R$ is convex and decreasing, then $m\circ g$ is convex. In particular:
    \begin{enumerate}[label=\textnormal{(\alph*)}]
        \item The function $d_K\circ g:X\to\R$ is convex.
        \item If $\lambda\in K_{\infty}^{\circ}$, then $x\mapsto\scal{\lambda,g(x)}$ is convex.
        \item The set $M=\{x\in X: g(x)\in K\}$ is convex.
    \end{enumerate}
\end{lem}
\begin{proof}
    Let $x,y\in X$ and $x_{\alpha}=\alpha x+(1-\alpha) y$, $\alpha\in[0,1]$. Then $g(x_{\alpha})\ge_K \alpha g(x)+(1-\alpha) g(y)$ by the concavity of $g$. Applying $m$ on both sides yields
    \begin{equation*}
        m(g(x_{\alpha}))\le m(\alpha g(x)+(1-\alpha)g(y))\le
        \alpha m(g(x))+(1-\alpha)m(g(y)),
    \end{equation*}
    where we used the monotonicity and the convexity of $m$. Hence, $m\circ g$ is convex. Assertion (a) now follows because $d_K$ is decreasing (see above) and convex \cite[Cor.\ 12.12]{Bauschke2011}. Similarly, for (b), the function $y\mapsto\scal{\lambda,y}$ with $\lambda\in K_{\infty}^{\circ}$ is obviously a convex function, and it is decreasing because $\scal{\lambda,k}\le 0$ for all $k\in K_{\infty}$. Finally, for (c), note that
    \begin{equation*}
        M=\{ x\in X: g(x)\in K \}=\{ x\in X: d_K(g(x))\le 0 \}.
    \end{equation*}
    Hence, $M$ is a lower level set of the convex function $d_K\circ g$ and therefore a convex set.
\end{proof}

\noindent
Note that the extreme case $K_{\infty}=\{0\}$ can occur, e.g.\ if $K$ is bounded. In this case, monotonicity becomes trivial and convexity and concavity reduce to linearity.

It is possible to characterize $K_{\infty}^{\circ}$ by means of the so-called barrier cone to $K$, see \cite{Bauschke2011}. Here, we will only need the following observation.

\begin{lem}\label{Lem:RecessionConePolar}
    If $y\in H$, then $y-P_K(y)\in K_{\infty}^{\circ}$.
\end{lem}
\begin{proof}
    Let $k\in K$ be fixed and let $z\in K_{\infty}$, $\alpha\ge 0$. Then $\alpha z\in K_{\infty}$ and therefore $\alpha z+k\in K$. A standard projection inequality yields $\scal{y-P_K(y),\alpha z+k-P_K(y)}\le 0$. But this clearly cannot hold for all $\alpha$ if $\scal{y-P_K(y),z}>0$. Hence, $\scal{y-P_K(y),z}\le 0$.
\end{proof}

\subsection{The KKT Conditions}

We now turn to the variational inequality \eqref{Eq:VI} and discuss its KKT conditions. Starting with this section, we assume that the mapping $F$ is continuously differentiable and that $g$ is twice continuously differentiable. Consider now the Lagrange function
\begin{equation}\label{Eq:L}
    \Lag: X\times H\to X^*, \quad \Lag (x,\lambda):=F(x)+g'(x)^* \lambda.
\end{equation}
Note that, if the VI originates from a minimization problem, then $\Lag$ is actually the derivative of the conventional Lagrange function. The following are the standard first-order conditions which we will use throughout this paper.

\begin{dfn}\label{Dfn:KKT}
    A tuple $(\bar{x},\bar{\lambda})\in X\times H$ is a \emph{KKT point} of \eqref{Eq:VI}, \eqref{Eq:M} if
    \begin{equation}\label{Eq:KKT}
        \Lag(\bar{x},\bar{\lambda})=0 \quad\text{and}\quad
        \bar{\lambda}\in\NorCone{K}{g(\bar{x})}.
    \end{equation}
    We call $\bar{x}\in X$ a stationary point if $(\bar{x},\bar{\lambda})$ is a KKT point for some $\bar{\lambda}\in H$, and denote by $\mathcal{M}(\bar{x})$ the corresponding set of multipliers.
\end{dfn}

\noindent
Note that $\bar{\lambda}\in\NorCone{K}{g(\bar{x})}$ implies $g(\bar{x})\in K$, since otherwise the normal cone would be empty. Moreover, we remark that, if $K$ is a cone, then $\bar{\lambda}\in\NorCone{K}{g(\bar{x})}$ is equivalent to the complementarity conditions $g(\bar{x})\in K$, $\bar{\lambda}\in K^{\circ}$, and $\scal{\bar{\lambda},g(\bar{x})}=0$, see \cite[Ex.\ 2.62]{Bonnans2000}.

The relationship between the VI and its KKT conditions is given as follows: if $\bar{x}$ solves the VI and a suitable constraint qualification holds in $\bar{x}$, then there exists a multiplier $\bar{\lambda}$ such that $(\bar{x},\bar{\lambda})$ is a KKT point \cite[Remark 5.8]{Bonnans2000}. On the other hand, it is easy to see that the KKT conditions are always sufficient for the VI \eqref{Eq:VI}, even if $M$ is nonconvex. This result is contained in the following theorem and crucially depends on the fact that the VI uses the tangent cone $\mathcal{T}_M$ and not $M$ itself.

\begin{thm}\label{Thm:SufficiencyKKT}
    If $(\bar{x},\bar{\lambda})$ is a KKT point of the VI, then $\bar{x}$ is a solution of the VI.
\end{thm}
\begin{proof}
    Let $(\bar{x},\bar{\lambda})$ be a KKT point and $d\in\TngCone{M}{\bar{x}}$. Then $d=\lim_{k\to\infty}(x^k-\bar{x})/t_k$ with $\{x^k\}\subseteq M$, $x^k\to\bar{x}$, and $t_k\downarrow 0$. Hence,
    \begin{equation*}
        \dual{F(\bar{x}),d}=\dual{-g'(\bar{x})^* \bar{\lambda},
            \lim_{k\to\infty} \frac{x^k-\bar{x}}{t_k}}=
        -\lim_{k\to\infty} \frac{1}{t_k} \scal{\bar{\lambda},g'(\bar{x})(x^k-\bar{x})}.
    \end{equation*}
    But $g'(\bar{x})(x^k-\bar{x})=g(x^k)-g(\bar{x})+o(t_k)$ and therefore
    \begin{equation*}
        \dual{F(\bar{x}),d}=-\lim_{k\to\infty}\frac{1}{t_k}
        \scal{\bar{\lambda},g(x^k)-g(\bar{x})}\ge 0,
    \end{equation*}
    where we used $\bar{\lambda}\in \NorCone{K}{g(\bar{x})}$ and $g(x^k)\in K$ for all $k$.
\end{proof}

\noindent
For a given KKT point $(\bar{x},\bar{\lambda})$ and $\eta\ge 0$, we define the extended critical cone
\begin{equation*}
    C_{\eta}(\bar{x}):=
    \bigl\{ d\in X: \dual{F(\bar{x}),d}\le \eta\|d\|_X,~
    g'(\bar{x})d\in\TngCone{K}{g(\bar{x})} \bigr\}.
\end{equation*}
The following is the second-order condition which we will use throughout this paper.

\begin{dfn}\label{Dfn:SOSC}
    Let $(\bar{x},\bar{\lambda})$ be a KKT point of the VI. We say that the \emph{second-order sufficient condition (SOSC)} holds in $(\bar{x},\bar{\lambda})$ if there are $\eta,c>0$ such that
    \begin{equation*}
        \dual{D_x \Lag(\bar{x},\bar{\lambda})d,d}\ge c\|d\|_X^2
        \quad\text{for all }d\in C_{\eta}(\bar{x}).
    \end{equation*}
\end{dfn}

\noindent
Note that we use the terminology ``second-order sufficient condition'' mainly for the sake of consistency with a similar condition from nonlinear optimization, e.g.\ \cite[Def.\ 3.60]{Bonnans2000}. For variational problems such as \eqref{Eq:VI}, there is actually no need for sufficiency conditions to complement the KKT system because the latter always implies that $\bar{x}$ is a solution of the VI (see Theorem \ref{Thm:SufficiencyKKT}).

Let us also note that Definition \ref{Dfn:SOSC} is slightly different from the second-order sufficient condition for nonlinear optimization \cite[Def.\ 3.60]{Bonnans2000} because our extended critical cone is slightly smaller. However, under the Robinson constraint qualification (see below and \cite[Def.\ 2.86]{Bonnans2000}), the corresponding second-order conditions coincide \cite[Remark 3.68]{Bonnans2000}. Moreover, and more importantly, our subsequent analysis will be based on \cite[Thm.\ 5.9]{Bonnans2000} which directly uses the ``smaller'' critical cone together with the following condition.

\begin{dfn}\label{Dfn:SRC}
    Let $(\bar{x},\bar{\lambda})$ be a KKT point of the VI, and $K_0:=\bigl\{ y\in K: \scal{\bar{\lambda},y-g(\bar{x})}=0 \bigr\}$. We say that the \emph{strict Robinson condition (SRC)} holds in $(\bar{x},\bar{\lambda})$ if
    \begin{equation}\label{Eq:SRC}
        0\in \operatorname{int}\bigl(g(\bar{x})+g'(\bar{x})X-K_0\bigr).
    \end{equation}
\end{dfn}

\noindent
Note that the standard Robinson constraint qualification arises if we replace $K_0$ in \eqref{Eq:SRC} by the larger set $K$. Hence, SRC is stronger than the Robinson constraint qualification and the equivalent regularity condition of Zowe and Kurcyusz \cite{Zowe1979}. On the other hand, SRC implies the uniqueness of $\bar{\lambda}$ and is weaker than the surjectivity of $g'(\bar{x})$, which is a typical regularity assumption for infinite-dimensional problems.

It should be noted that the definition of SRC presupposes the existence of $\bar{\lambda}$ and therefore depends not only on the constraints but also on the function $F$. Hence, we refrain from calling \eqref{Eq:SRC} a constraint qualification (in contrast to \cite{Bonnans2000}, where SRC is called the \emph{strict constraint qualification}). A similar condition which is occasionally used in the finite-dimensional literature is the strict Mangasarian-Fromovitz condition \cite{Birgin2012,Floudas2001,Kyparisis1985}. This condition turns out to be a special case of SRC \cite[Remark 4.49]{Bonnans2000} and is also not a constraint qualification \cite{Wachsmuth2013}.

\section{Error Bounds for the Variational Problem}\label{Sec:ErrorBounds}

Recall that the KKT conditions of the VI are given by
\begin{equation*}
    \Lag(\bar{x},\bar{\lambda})=0 \quad\text{and}\quad
    \bar{\lambda}\in\NorCone{K}{g(\bar{x})},
\end{equation*}
where $(\bar{x},\bar{\lambda})\in X\times H$. The last condition is well-known \cite[Prop.\ 6.46]{Bauschke2011} to be equivalent to $g(\bar{x})=P_K(g(\bar{x})+\bar{\lambda})$. This suggests defining the residual mapping
\begin{equation}\label{Eq:KKTResidual}
    \sigma(x,\lambda):=\|\Lag(x,\lambda)\|_{X^*}
    +\|g(x)-P_K(g(x)+\lambda)\|_H.
\end{equation}
Clearly, the KKT conditions of the VI are equivalent to $\sigma(\bar{x},\bar{\lambda})=0$. We will use this relationship to construct suitable error bounds for the primal-dual variables.

In order to establish the error bound we are looking for, we first need a characterization of local error bounds in terms of a local upper Lipschitz property (or calmness) of the KKT system. This result has appeared in various forms in the literature \cite{Ding2017,Fischer2002,Izmailov2012a}, albeit mostly in a finite-dimensional setting. In our notation, it involves certain perturbations of the KKT system \eqref{Eq:KKT} with a parameter pair $p=(\alpha,\beta)\in X^*\times H$. Without loss of generality, we equip this product space with the norm $\|(\alpha,\beta)\|_{X^*\times H}:=\|\alpha\|_{X^*}+\|\beta\|_H$. Recall also that $\mathcal{M}(\bar{x})$ denotes the set of Lagrange multipliers corresponding to $\bar{x}$.

\begin{thm}\label{Thm:ErrorBoundEquivalence}
    Let $(\bar{x},\bar{\lambda})\in X\times H$ be a KKT point of the VI. Then the following assertions are equivalent:
    \begin{enumerate}[label=\textnormal{(\alph*)}]
        \item There are a neighborhood $U$ of $\bar{x}$ and $c>0$ such that, for all $p=(\alpha,\beta)\in X^*\times H$ close to $(0,0)$, any solution $(x_p,\lambda_p)\in U\times H$ of the perturbed KKT system
        \begin{equation}\label{Eq:PerturbedKKT}
            \Lag(x,\lambda)=\alpha, \quad \lambda\in \NorCone{K}{g(x)-\beta}
        \end{equation}
        satisfies the estimate $\|x_p-\bar{x}\|_X+\dist{\lambda_p,\mathcal{M}(\bar{x})}
        \le c\|p\|_{X^*\times H}$.
        \item There are a neighborhood $U$ of $\bar{x}$ and $c>0$ such that, for all $(x,\lambda)\in U\times H$ with $\sigma(x,\lambda)$ sufficiently small,
        \begin{equation*}
            \|x-\bar{x}\|_X+\dist{\lambda,\mathcal{M}(\bar{x})}
            \le c \sigma(x,\lambda).
        \end{equation*}
    \end{enumerate}
\end{thm}
\begin{proof}
    (b)$\Rightarrow$(a): Let $p=(\alpha,\beta)\in X^*\times H$. It is an easy consequence of \cite[Cor.\ 4.10]{Bauschke2011} that the mapping $y\mapsto y-P_K(y+\lambda_p)$ is nonexpansive. Hence, we obtain the inequality
    \begin{equation*}
        \|g(x_p)-P_K(g(x_p)+\lambda_p)\|_H
        \le \|\beta\|_H + \|g(x_p)-\beta-P_K(g(x_p)
        -\beta+\lambda_p)\|_H.
    \end{equation*}
    Since $\lambda_p\in\NorCone{K}{g(x_p)-\beta}$, the last term is equal to zero \cite[Prop.\ 6.46]{Bauschke2011} and we obtain $\sigma(x_p,\lambda_p)\le\|\alpha\|_{X^*}+\|\beta\|_H=\|p\|_{X^*\times H}$. Choosing $p=(\alpha,\beta)$ sufficiently close to $0$, we see that $\sigma(x_p,\lambda_p)$ becomes arbitrarily small. Hence, we can apply (b) and obtain
    \begin{equation*}
        \|x_p-\bar{x}\|_X+\dist{\lambda_p,\mathcal{M}(\bar{x})}
        \le c\sigma(x_p,\lambda_p)\le c\|p\|_{X^*\times H}.
    \end{equation*}
    
    \noindent
    (a)$\Rightarrow$(b): Shrinking $U$ if necessary, we may assume that $\|g'(x)^*\|_{\mathcal{L}(H,X^*)}\le c_1$ for all $x\in U$ with some constant $c_1\ge 0$. Let $(x,\lambda)\in U\times H$, set $\delta:=\sigma(x,\lambda)$, and define
    \begin{equation*}
        \hat{g}:=P_K(g(x)+\lambda),\quad \hat{\lambda}:=g(x)+\lambda-\hat{g}.
    \end{equation*}
    Now, let $\alpha:=\Lag(x,\hat{\lambda})$ and $\beta:=g(x)-\hat{g}$. Then $\hat{\lambda}\in\NorCone{K}{\hat{g}}$ and, hence, $(x,\hat{\lambda})$ solves the perturbed KKT system corresponding to $\sigma:=(\alpha,\beta)$. Moreover, we have $\|\beta\|_H=\|\hat{g}-g(x)\|_H=\|g(x)-P_K(g(x)+\lambda)\|_H\le\delta$ and $\|\hat{\lambda}-\lambda\|_H=\|\beta\|_H\le\delta$. This implies
    \begin{equation*}
        \|\sigma\|_{X^*\times H}=\|\Lag(x,\hat{\lambda})\|_{X^*}+\|\beta\|_H\le
        \|\Lag(x,\lambda)\|_{X^*}+(c_1+1)\|\beta\|_H \le (c_1+2)\delta.
    \end{equation*}
    Hence, if $\delta=\sigma(x,\lambda)$ is small enough, then $\sigma$ becomes arbitrarily close to $0$. We can therefore apply (a) to $(x,\hat{\lambda})$ and obtain
    \begin{equation*}
        \|x-\bar{x}\|_X+\dist{\hat{\lambda},\mathcal{M}(\bar{x})}
        \le c\|\sigma\|_{X^*\times H}\le c(c_1+2)\delta.
    \end{equation*}
    But $\|\hat{\lambda}-\lambda\|_H\le\delta$ and, hence, $\dist{\hat{\lambda},\mathcal{M}(\bar{x})}\ge \dist{\lambda,\mathcal{M}(\bar{x})}-\delta$ by the nonexpansiveness of the distance function. This finally yields
    \begin{equation*}
        \|x-\bar{x}\|_X+\dist{\lambda,\mathcal{M}(\bar{x})}\le \bigl[c(c_1+2)+1\bigr]\delta,
    \end{equation*}
    and the proof is complete.
\end{proof}

\noindent
Let us stress that the distance estimate provided by the above theorem holds if $x$ is close to $\bar{x}$; in particular, no assumption on the proximity of $\lambda$ to $\mathcal{M}(\bar{x})$ is necessary. We also remark that (a) does not make any assertion about the existence of solutions to the perturbed KKT conditions \eqref{Eq:PerturbedKKT}. These may have solutions for some but not all $\sigma$.

Theorem \ref{Thm:ErrorBoundEquivalence} is our main tool for establishing local error bounds for the distance of $(x,\lambda)$ to the primal-dual solution set in terms of the residual mapping $\sigma$. To verify such an error bound, we only need to prove property (a) of the theorem. The following result does precisely that and is based on the perturbation theory from \cite{Bonnans2000}.

\begin{thm}\label{Thm:ErrorBound}
    Assume that $(\bar{x},\bar{\lambda})$ is a KKT point which satisfies SOSC and the strict Robinson condition. Then $\mathcal{M}(\bar{x})=\{\bar{\lambda}\}$ and there is a $c>0$ such that, for all $(x,\lambda)\in X\times H$ with $x$ sufficiently close to $\bar{x}$ and $\sigma(x,\lambda)$ sufficiently small,
    \begin{equation}\label{Eq:ErrorBound}
        \|x-\bar{x}\|_X+\|\lambda-\bar{\lambda}\|_H
        \le c \sigma(x,\lambda).
    \end{equation}
\end{thm}
\begin{proof}
    The uniqueness of $\bar{\lambda}$ follows as in \cite[Prop.\ 4.47]{Bonnans2000}, see also the discussion in Section 5.1.2 of that reference. For the error bound result, we essentially need to apply \cite[Thm.\ 5.9]{Bonnans2000} and Theorem \ref{Thm:ErrorBoundEquivalence}. Since some technical details need to be considered, we give a formal proof here. To this end, assume that the error bound in question does not hold. Then property (a) from Theorem \ref{Thm:ErrorBoundEquivalence} does not hold either; hence, there are sequences $x^k\to\bar{x}$, $\{\lambda^k\}\subseteq H$ and $\{\sigma^k\}\subseteq X^*\times H$ with $\sigma^k=(\alpha^k,\beta^k)\to 0$ such that, for all $k$, $(x^k,\lambda^k)$ satisfies the perturbed KKT conditions \eqref{Eq:PerturbedKKT} corresponding to $\sigma^k$, and
    \begin{equation}\label{Eq:CorErrorBound1}
        \|x^k-\bar{x}\|_X+\|\lambda^k-\bar{\lambda}\|_H \ge k \|\sigma^k\|_{X^*\times H}.
    \end{equation}
    Now, let $F(x,\sigma):=F(x)-\alpha$ and $G(x,\sigma):=g(x)-\beta$ for $\sigma=(\alpha,\beta)\in X^*\times H$. Then $(x^k,\lambda^k)$ satisfies
    \begin{equation*}
        F(x^k,\sigma^k)+D_x G(x^k,\sigma^k)^* \lambda^k=0, \quad
        \lambda^k\in \NorCone{K}{G(x^k,\sigma^k)}
    \end{equation*}
    for all $k$. Applying \cite[Thm.\ 5.9]{Bonnans2000} yields a contradiction to \eqref{Eq:CorErrorBound1}.
\end{proof}

\noindent
The function $\sigma$ is locally Lipschitz-continuous with respect to $(x,\lambda)$, and globally so with respect to $\lambda$. Hence, we can extend the one-sided error bound \eqref{Eq:ErrorBound} to
\begin{equation}\label{Eq:DoubleErrorBound}
    c_1 \sigma(x,\lambda) \le \|x-\bar{x}\|_X+\|\lambda-\bar{\lambda}\|_H
    \le c_2 \sigma(x,\lambda)
\end{equation}
for suitable constants $c_1,c_2>0$ and all $(x,\lambda)\in X\times H$ with $x$ near $\bar{x}$.

For certain problem classes, it is possible to establish error bounds under weaker assumptions than those given above. The most important example in this direction is if the set $K$ is (generalized) polyhedral, e.g.\ in nonlinear programming. Roughly speaking, one can use Hoffman's lemma \cite[Thm.\ 2.200]{Bonnans2000} to get the ``dual part'' of the error bound for free, while the primal part again follows from SOSC. As a result, one obtains a primal-dual error bound under SOSC alone (with the restriction that the multiplier is not necessarily unique). Unsurprisingly, this result does not extend to the non-polyhedral case, which shows that additional assumptions such as SRC are inevitable.
\begin{exm}
    Let $X:=H:=\ell^2(\R)$ be the space of square-summable real sequences. Consider the optimization problem \eqref{Eq:Opt}, \eqref{Eq:M} with $f(x):=\|x\|_X^2/2$, $g(x):=(x_i/i)_{i=1}^{\infty}$, and $K$ the nonnegative cone in $X$. It is easy to see that $(\bar{x},\bar{\lambda}):=(0,0)$ is the unique KKT point of this problem, and that SOSC holds. Now, let $x^k:=e^k/k$ and $\lambda^k:=-e^k$, where $\{e^k\}$ is the sequence of unit vectors. Then
    \begin{equation*}
        \sigma(x^k,\lambda^k)=\|\Lag(x^k,\lambda^k)\|_{X^*}+\|g(x^k)-P_K(g(x^k)+\lambda^k)\|_H
        =k^{-2}
    \end{equation*}
    for all $k$. Moreover, $x^k\to \bar{x}$, but $\lambda^k\not\to \bar{\lambda}$. Hence, a local error bound does not hold. (In particular, SRC cannot hold, even though the Lagrange multiplier is actually unique.) A slightly different example is obtained by setting $\hat{x}^k:=e^k/k^2$ and $\hat{\lambda}^k:=-e^k/k$. In this case, $(\hat{x}^k,\hat{\lambda}^k)\to (\bar{x},\bar{\lambda})$, but an easy calculation shows that
    \begin{equation*}
        \sigma(\hat{x}^k,\hat{\lambda}^k)=k^{-3} \quad\text{and}\quad
        \|\hat{x}^k-\bar{x}\|_X+\|\hat{\lambda}^k-\bar{\lambda}\|_H=k^{-2}+k^{-1}.
    \end{equation*}
    In particular, the error bound is violated even if the multiplier is close to $\bar{\lambda}$.
\end{exm}

\noindent
We close this section by noting that the error bound in Theorem~\ref{Thm:ErrorBound} necessarily implies that the Lagrange multiplier $\bar{\lambda}$ is unique. It is natural to ask whether sufficient conditions can be established which guarantee the error bound property with a nonunique multiplier (as in the statement of Theorem~\ref{Thm:ErrorBoundEquivalence}). However, it turns out that the resulting conditions are often of technical nature, see \cite[Thm.~4.51]{Bonnans2000}, and not easily verified for common problem classes. Therefore, and since the case covered by Theorem~\ref{Thm:ErrorBound} suffices for our applications, we restrict ourselves to the situation where $\bar{\lambda}$ is unique.

\section{The Augmented Lagrangian Method}\label{Sec:Method}

We now present the augmented Lagrangian method for the variational inequality \eqref{Eq:VI}. The main approach is to penalize the function $g$ and therefore reduce the VI to a sequence of (unconstrained) nonlinear equations. Consider the augmented Lagrangian
\begin{equation}\label{Eq:AL}
   \Lag_{\rho}:X\times H\to X^*, \quad \Lag_{\rho}(x,\lambda):=
   F(x)+\rho g'(x)^* \mleft[ g(x)+\frac{\lambda}{\rho}-P_K\mleft(
   g(x)+\frac{\lambda}{\rho}\mright) \mright].
\end{equation}
Note that, if $K$ is a cone, then we can simplify the above formula to $\Lag_{\rho}(x,\lambda)=F(x)+g'(x)^* P_{K^{\circ}}(\lambda+\rho g(x))$ by using Moreau's decomposition \cite{Bauschke2011,Moreau1962}.

For the construction of our algorithm, we will need a means of controlling the penalty parameters. To this end, we define the utility function
\begin{equation}\label{Eq:V}
   V(x,\lambda,\rho):=\|\Lag_{\rho}(x,\lambda)\|_{X^*}+
   \left\|g(x)-P_K\mleft(g(x)+\frac{\lambda}{\rho}\mright)\right\|_H.
\end{equation}
This function requires some elaboration. The first term in \eqref{Eq:V} measures the precision with which the subproblem was solved in the current iteration. The second term is a composite measure of feasibility and complementarity; it arises from an inherent slack variable transformation which is often used to define the augmented Lagrangian for inequality or cone constraints. As a result, the function $V$ measures optimality, feasibility and complementarity at the current iterate.

\begin{alg}[Augmented Lagrangian method]\label{Alg:ALM}\leavevmode
\begin{itemize}[font=\normalfont]
    \item[(S.0)] Let $(x^0,\lambda^0)\in X\times H$, $B\subseteq 
        H$ bounded, $\rho_0>0$, $\gamma>1$, $\tau\in(0,1)$, 
        and set $k:=0$.
    \item[(S.1)] If $(x^k,\lambda^k)$ satisfies a suitable termination 
        criterion: STOP.
    \item[(S.2)] Choose $\BddMul^k\in B$ and compute an inexact zero (see below)
        $x^{k+1}$ of $\Lag_{\rho_k}(\cdot,\BddMul^k)$.
    \item[(S.3)] Update the vector of multipliers to
        \begin{equation}\label{Eq:MultUpdate}
            \lambda^{k+1}:=\rho_k \mleft[ g(x^{k+1})+\frac{\BddMul^k}{\rho_k}
            -P_K\mleft(g(x^{k+1})+\frac{\BddMul^k}{\rho_k}\mright) \mright].
        \end{equation}
    \item[(S.4)] If $k=0$ or
        \begin{equation}\label{Eq:RhoTest}
            V(x^{k+1},\BddMul^k,\rho_k)\le \tau V(x^k,\BddMul^{k-1},\rho_{k-1})
        \end{equation}
        holds, set $\rho_{k+1}:=\rho_k$; otherwise, set $\rho_{k+1}:=\gamma\rho_{k}$.
    \item[(S.5)] Set $k\leftarrow k+1$ and go to \textnormal{(S.1)}.
\end{itemize}
\end{alg}

\noindent
Let us make some simple observations. First, regardless of the primal iterates $\{x^k\}$, the multipliers $\{\lambda^k\}$ always lie in the polar cone $K_{\infty}^{\circ}$ by Lemma \ref{Lem:RecessionConePolar}. Moreover, if $K$ is a cone, then the Moreau decomposition \cite{Moreau1962} implies that $\lambda^{k+1}=P_{K^{\circ}}(\BddMul^k+\rho_k g(x^{k+1}))$.

Secondly, we note that Algorithm \ref{Alg:ALM} uses a safeguarded multiplier sequence $\{\BddMul^k\}$ in certain places where classical augmented Lagrangian methods use the sequence $\{\lambda^k\}$. This bounding scheme goes back to \cite{Andreani2007,Pang2005} and is crucial to establishing strong global convergence results for the method \cite{Andreani2007,Birgin2010,Birgin2014,Kanzow2016}. In practice, one usually tries to keep $\BddMul^k$ as ``close'' as possible to $\lambda^k$, e.g.\ by defining $\BddMul^k:=P_B(\lambda^k)$, where $B$ (the bounded set from the algorithm) is chosen suitably to allow cheap projections.

The third observation is that if the sequence of penalty parameters $\{\rho_k\}$ remains bounded, then \eqref{Eq:RhoTest} yields $V(x^{k+1},\BddMul^k,\rho_k)\to 0$. In this case, the definition of $V$ implies that both the residual $\|\Lag_{\rho_k}(x^{k+1},\BddMul^k)\|_{X^*}$ of the subproblems and the composite feasibility-complementarity measure converge to zero. Hence, from a theoretical point of view, the case of bounded $\{\rho_k\}$ is the ``good'' case. In Section \ref{Sec:LocalConv}, we will actually prove the boundedness of $\{\rho_k\}$ under certain assumptions, and this result crucially depends on the fact that the function $V$ involves both terms from \eqref{Eq:V}.

For the remainder of this paper, we make the following assumption.

\begin{asm}\label{Asm:Subproblems}
    There is a null sequence $\{\varepsilon_k\}\subseteq [0,\infty)$ such that
    \begin{equation*}
        \|\Lag_{\rho_k}(x^{k+1},\BddMul^k)\|_{X^*}\le \varepsilon_{k+1}
        \quad\text{for all }k.
    \end{equation*}
\end{asm}

\noindent
This assumption is fairly natural and basically asserts that $x^{k+1}$ is an approximate zero point of $\Lag_{\rho_k}(\cdot,\BddMul^k)$, and that the degree of inexactness vanishes as $k\to\infty$.

\section{Global Convergence}\label{Sec:GlobalConv}

In this section, we discuss the global convergence properties of Algorithm \ref{Alg:ALM}. Some general results in this direction were obtained in \cite{Kanzow2017a,Kanzow2016} for optimization and generalized Nash equilibrium problems by assuming that the sequence $\{x^k\}$ has a limit point which satisfies a suitable constraint qualification.

Here, we pursue a slightly different approach. Since the constraints occurring in VIs are often convex, we can use this convexity to directly show that (weak) limit points are solutions of the VI. This idea has the advantage that we do not need any constraint qualification (in return, we do not get much information on the sequence $\{\lambda^k\}$).

Recall that we have already assumed $F$ to be continuously differentiable and $g$ twice continuously differentiable (for this section, one degree less would actually be sufficient). We now make the following additional assumptions.

\begin{asm}\label{Asm:GeneralConv}
    We assume that $g$ is concave with respect to $K_{\infty}$ (see Section \ref{Sec:Prelims}) and that $\dual{F(x),x-y}$ is weakly sequentially lsc with respect to $x$ for all $y\in X$.
\end{asm}

\noindent
The first of the above conditions ensures the convexity of the set $M$, see Lemma \ref{Lem:GeneralizedConvexity}. The second assumption implies, roughly speaking, that weak limit points of a sequence of ``approximate solutions'' of the VI are exact solutions. Note that this condition has also been used in certain existence results for VIs \cite{Isac1992}.

\begin{lem}\label{Lem:Feasibility}
    Let Assumptions \ref{Asm:Subproblems}, \ref{Asm:GeneralConv} hold, and let $\bar{x}$ be a weak limit point of $\{x^k\}$. Then $\bar{x}$ is a minimizer of the convex function $d_K\circ g$. In particular, if the feasible set $M$ is nonempty, then $\bar{x}$ is feasible.
\end{lem}
\begin{proof}
    Note that the function $d_K\circ g$ is convex by Lemma \ref{Lem:GeneralizedConvexity} and continuous, hence weakly sequentially lower semicontinuous \cite[Thm.\ 9.1]{Bauschke2011}. If $\{\rho_k\}$ remains bounded, then the penalty updating scheme \eqref{Eq:RhoTest} implies
    \begin{equation*}
        d_K(g(x^{k+1})) \le \mleft\| g(x^{k+1})-P_K \mleft( g(x^{k+1})+\frac{\BddMul^k}{\rho_k}
        \mright) \mright\|_H \le V(x^{k+1},\BddMul^k,\rho_k)\to 0
    \end{equation*}
    and therefore $d_K(g(\bar{x}))=0$. We now assume that $\rho_k\to\infty$ and define the auxiliary functions $h_k(x)=d_K^2(g(x)+\BddMul^k/\rho_k)$. Note that $h_k$ is continuously differentiable \cite[Cor.\ 12.30]{Bauschke2011}. Let $x^{k+1}\wto_{\mathcal{K}}\bar{x}$ for some index set $\mathcal{K}\subseteq \N$ and assume that there is a point $y\in X$ with $d_K(g(y))<d_K(g(\bar{x}))$. The weak sequential lower semicontinuity of $d_K\circ g$ and the boundedness of $\{\BddMul^k\}$ imply that
    \begin{equation*}
        \liminf_{k\in\mathcal{K}} h_k(x^{k+1})
        =\liminf_{k\in\mathcal{K}} d_K^2 \bigl(g(x^{k+1})+\BddMul^k/\rho_k\bigr)
        \ge d_K^2(g(\bar{x}))
    \end{equation*}
    and $h_k(y)\to d_K^2(g(y))$. Hence, there is a constant $c_1>0$ such that $h_k(x^{k+1})-h_k(y)\ge c_1$ for all $k\in\mathcal{K}$ sufficiently large. Since $h_k$ is convex by Lemma \ref{Lem:GeneralizedConvexity}, it follows that
    \begin{equation}\label{Eq:LemFeasibility1}
        \dual{h_k'(x^{k+1}),y-x^{k+1}}\le h_k(y)-h_k(x^{k+1})\le -c_1
    \end{equation}
    for all $k\in\mathcal{K}$ sufficiently large. Now, let $\{\varepsilon_k\}$ be the sequence from Assumption \ref{Asm:Subproblems}. Using \cite[Cor.\ 12.30]{Bauschke2011} for the derivative of $h_k$, we obtain
    \begin{align*}
        -\varepsilon_{k+1} \|y-x^{k+1}\|_X & \le \dual{\Lag_{\rho_k}(x^{k+1},\BddMul^k),y-x^{k+1}} \\
        & = \dual{F(x^{k+1}),y-x^{k+1}}+\frac{\rho_k}{2}\dual{h_k'(x^{k+1}),y-x^{k+1}}.
    \end{align*}
    By Assumption \ref{Asm:GeneralConv}, the function $\dual{F(x),x-y}$ is weakly sequentially lsc with respect to $x$. Hence, there is a constant $c_2\in\R$ such that $\dual{F(x^{k+1}),y-x^{k+1}}\le c_2$ for all $k\in\mathcal{K}$. This together with \eqref{Eq:LemFeasibility1} implies
    \begin{equation*}
        -\varepsilon_{k+1} \|y-x^{k+1}\|_X\le c_2-\frac{\rho_k c_1}{2}\to -\infty.
    \end{equation*}
    Since $\{x^{k+1}\}_{\mathcal{K}}$ is bounded and $\varepsilon_k\to 0$, this is a contradiction.
\end{proof}

\noindent
Note that Lemma \ref{Lem:Feasibility} guarantees that every weak limit point $\bar{x}$ automatically minimizes the constraint violation even if the feasible set $M$ is empty.

We now prove the optimality of limit points. To this end, we first need a technical lemma which essentially asserts some sort of ``approximate normality'' of $\lambda^{k+1}$ with respect to $K$ and $g(x^{k+1})$, the latter not necessarily being an element of $K$. Note that the result does not require any assumptions but directly follows from the definition of $\lambda^{k+1}$ as well as the updating scheme \eqref{Eq:RhoTest}.

\begin{lem}\label{Lem:Limsup}
    We have $\limsup_{k\to\infty}\scal{\lambda^{k+1},y-g(x^{k+1})}\le 0$ for all $y\in K$.
\end{lem}
\begin{proof}
    Let $y\in K$ and define the sequence $s^{k+1}:=P_K(g(x^{k+1})+\BddMul^k/\rho_k)$. Then $s^{k+1}\in K$ and it follows from \cite[Prop.\ 6.46]{Bauschke2011} that $\lambda^{k+1}\in\NorCone{K}{s^{k+1}}$. Moreover, we have
    \begin{equation}\label{Eq:LemLimsup1}
        g(x^{k+1})=\frac{\lambda^{k+1}-\BddMul^k}{\rho_k}+s^{k+1}.
    \end{equation}
    This yields
    \begin{align}
        \scal{\lambda^{k+1},y-g(x^{k+1})} & =
        \scal{\lambda^{k+1},y-\frac{1}{\rho_k}(\lambda^{k+1}-\BddMul^k)-s^{k+1}} \notag \\
        & \le \frac{1}{\rho_k} \Bigl[\scal{\lambda^{k+1},\BddMul^k}-\|\lambda^{k+1}\|_H^2\Bigr]
        \label{Eq:LemLimsup2},
    \end{align}
    where we used $\lambda^{k+1}\in\NorCone{K}{s^{k+1}}$ for the last inequality. Now, if $\{\rho_k\}$ is bounded, then \eqref{Eq:RhoTest} and \eqref{Eq:LemLimsup1} imply $\|\lambda^{k+1}-\BddMul^k\|_H/\rho_k\to 0$ and therefore $\|\lambda^{k+1}-\BddMul^k\|_H\to 0$. This yields the boundedness of $\{\lambda^{k+1}\}$ in $H$ as well as $\scal{\lambda^{k+1},\BddMul^k}-\|\lambda^{k+1}\|_H^2=\scal{\lambda^{k+1},\BddMul^k-\lambda^{k+1}}\to 0$. Hence, the desired results follows from \eqref{Eq:LemLimsup2}. We now assume that $\rho_k\to\infty$. Note that \eqref{Eq:LemLimsup2} is a quadratic function in $\lambda$. A simple calculation therefore shows that
    \begin{equation*}
        \scal{\lambda^{k+1},y-g(x^{k+1})}\le \frac{1}{4\rho_k}\|\BddMul^k\|_H^2.
    \end{equation*}
    The boundedness of $\{\BddMul^k\}$ now implies $\limsup_{k\to\infty} \scal{\lambda^{k+1},y-g(x^{k+1})}\le 0$.
\end{proof}

\noindent
The above result can be stated more concisely if $K$ is a cone. By inserting $0\in K$ into the inequality, it is easy to see that it is equivalent to $\liminf_{k\to\infty} \scal{\lambda^{k+1},g(x^{k+1})}\ge 0$.

We now turn to the main global convergence result.

\begin{thm}\label{Thm:Optimality}
    Let Assumptions \ref{Asm:Subproblems}, \ref{Asm:GeneralConv} hold, and let $\bar{x}$ be a weak limit point of $\{x^k\}$. If the feasible set $M$ is nonempty, then $\bar{x}$ is feasible and solves the VI.
\end{thm}
\begin{proof}
    Let $x^{k+1}\wto_{\mathcal{K}}\bar{x}$ for some $\mathcal{K}\subseteq\N$. The feasibility claim follows from Lemma \ref{Lem:Feasibility}. For the optimality, let $y\in M$ be any feasible point. Then $\dual{\Lag_{\rho_k}(x^{k+1},\BddMul^k),y-x^{k+1}}\ge -\varepsilon_{k+1} \|y-x^{k+1}\|_X$ by Assumption \ref{Asm:Subproblems} and, since $\Lag_{\rho_k}(x^{k+1},\BddMul^k)=\Lag(x^{k+1},\lambda^{k+1})$, we get
    \begin{align*}
        -\varepsilon_{k+1} \|y-x^{k+1}\|_X & \le
        \dual{F(x^{k+1})+g'(x^{k+1})^* \lambda^{k+1},y-x^{k+1}} \\
        & \le \dual{F(x^{k+1}),y-x^{k+1}}+\scal{\lambda^{k+1},g(y)-g(x^{k+1})},
    \end{align*}
    where we used the fact that $x\mapsto \scal{\lambda^{k+1},g(x)}$ is convex by Lemma \ref{Lem:GeneralizedConvexity} (recall that $\lambda^{k+1}\in K_{\infty}^{\circ}$). Using $\varepsilon_k\to 0$ and Lemma \ref{Lem:Limsup}, we now obtain $\liminf_{k\in\mathcal{K}}\dual{F(x^{k+1}),y-x^{k+1}}\ge 0$. Since $\dual{F(x),x-y}$ is weakly sequentially lsc, this implies $\dual{F(\bar{x}),y-\bar{x}}\ge 0$.
\end{proof}

\section{Local Convergence}\label{Sec:LocalConv}

We will now consider the local convergence characteristics of Algorithm \ref{Alg:ALM}. A key ingredient is the error bound property from Section \ref{Sec:ErrorBounds} which allows us to estimate the distance from $(x^k,\lambda^k)$ to $(\bar{x},\bar{\lambda})$ by using the function $\sigma$ from \eqref{Eq:KKTResidual}.

\begin{lem}\label{Lem:ConvSOSC}
    Let Assumption \ref{Asm:Subproblems} hold and let $(\bar{x},\bar{\lambda})$ be a KKT point satisfying the error bound \eqref{Eq:ErrorBound}. Then there is an $r>0$ such that, if $x^k\in B_r(\bar{x})$ for all $k$ and $d_K(g(x^k))\to 0$, then $(x^k,\lambda^k)\to (\bar{x},\bar{\lambda})$.
\end{lem}
\begin{proof}
    By Assumption \ref{Asm:Subproblems}, we have $\Lag(x^{k+1},\lambda^{k+1})=\Lag_{\rho_k}(x^{k+1},\BddMul^k)\to 0$. Hence, in view of the error bound property, it suffices to show that $g(x^{k+1})-P_K(g(x^{k+1})+\lambda^{k+1})\to 0$. To this end, define the sequence $s^{k+1}:=P_K(g(x^{k+1})+\BddMul^k/\rho_k)$. Then $s^{k+1}\in K$ and, as noted before, $\lambda^{k+1}\in\NorCone{K}{s^{k+1}}$ for all $k$. We now use the fact that $y\mapsto y-P_K(y+\lambda^{k+1})$ is nonexpansive, which is an easy consequence of \cite[Cor.\ 4.10]{Bauschke2011}. Therefore, the inverse triangle inequality yields
    \begin{equation}\label{Eq:LemConvSOSC1}
    \begin{aligned}
        \lefteqn{\|g(x^{k+1})-P_K(g(x^{k+1})+\lambda^{k+1})\|_H} & \\
        & \le \|g(x^{k+1})-s^{k+1}\|_H+\|s^{k+1}-P_K(s^{k+1}+\lambda^{k+1})\|_H.
    \end{aligned}
    \end{equation}
    The last term is equal to zero since $\lambda^{k+1}\in\NorCone{K}{s^{k+1}}$, cf.\ \cite[Cor.\ 6.46]{Bauschke2011}. Hence, to complete the proof, we only need to show that $\|s^{k+1}-g(x^{k+1})\|_H\to 0$. If $\{\rho_k\}$ is bounded, then this readily follows from the penalty updating scheme \eqref{Eq:RhoTest}. On the other hand, if $\rho_k\to\infty$, then
    \begin{equation*}
        \|s^{k+1}-g(x^{k+1})\|_H \le \|s^{k+1}-P_K(g(x^{k+1}))\|_H + d_K(g(x^{k+1}))\to 0,
    \end{equation*}
    where we used the nonexpansiveness of the projection operator.
\end{proof}

\noindent
The above lemma gives us some information about the behavior of zeros of the augmented Lagrangian in a neighborhood of $\bar{x}$. Note that the assumption $d_K(g(x^{k+1}))\to 0$ asserts that the iterates become (asymptotically) feasible and is often satisfied in practice, see also Lemma \ref{Lem:Feasibility}. For the remaining analysis, we now make the following assumption.

\begin{asm}\label{Asm:LocalConv}
    We assume that $(\bar{x},\bar{\lambda})$ is a KKT point of the VI which satisfies the local error bound \eqref{Eq:ErrorBound}. Moreover, the sequence $\{(x^k,\lambda^k)\}$ from Algorithm \ref{Alg:ALM} converges strongly to $(\bar{x},\bar{\lambda})$, and we have $\BddMul^k=\lambda^k$ for all $k$ sufficiently large.
\end{asm}

\noindent
One of the above assumptions which might require some elaboration is $\BddMul^k=\lambda^k$ for all $k$. The boundedness of $\{\BddMul^k\}$ is key to establishing global convergence of the algorithm, see Section \ref{Sec:GlobalConv}. Since $\lambda^k\to\bar{\lambda}$ in our setting, we do not need to force boundedness of $\{\BddMul^k\}$ and can simply set $\BddMul^k:=\lambda^k$ for all $k$. (In the context of Algorithm \ref{Alg:ALM}, we formally need to choose the bounded set $B$ sufficiently large to allow this.)

We will now prove convergence rates for the primal-dual sequence $\{(x^k,\lambda^k)\}$. Since the distance of $(x^k,\lambda^k)$ to $(\bar{x},\bar{\lambda})$ admits both upper and lower estimates relative to the residual terms $\sigma_k:=\sigma(x^k,\lambda^k)$ by \eqref{Eq:DoubleErrorBound}, we will largely base our analysis on the sequence $\{\sigma_k\}$, and the results on the primal-dual sequence $\{(x^k,\lambda^k)\}$ will follow directly.

\begin{lem}\label{Lem:SigmaConv}
    Let Assumptions \ref{Asm:Subproblems}, \ref{Asm:LocalConv} hold, and let $\sigma_k:=\sigma(x^k,\lambda^k)$. Then there is a constant $c_1>0$ such that
    \begin{equation*}
        \mleft(1-\frac{c_1}{\rho_k}\mright) \sigma_{k+1}
        \le \varepsilon_{k+1} + \frac{c_1}{\rho_k}\sigma_k
    \end{equation*}
    for all $k\in\N$ sufficiently large.
\end{lem}
\begin{proof}
    Observe that $\Lag_{\rho_k}(x^{k+1},\BddMul^k)=\Lag(x^{k+1},\lambda^{k+1})$ for all $k$. By Assumption \ref{Asm:Subproblems} and the definition of $\sigma_k$, we therefore have
    \begin{equation}\label{Eq:LemSigmaConv1}
        \sigma_{k+1}\le \varepsilon_{k+1}+\|g(x^{k+1})-P_K(g(x^{k+1})+\lambda^{k+1})\|_H.
    \end{equation}
    Now, let $k\in\N$ be large enough so that $\BddMul^k=\lambda^k$. Consider again the sequence $s^{k+1}:=P_K(g(x^{k+1})+\lambda^k/\rho_k)$. Using \eqref{Eq:LemConvSOSC1}, we see that
    \begin{equation}\label{Eq:LemSigmaConv2}
        \|g(x^{k+1})-P_K(g(x^{k+1})+\lambda^{k+1})\|_H \le
        \|g(x^{k+1})-s^{k+1}\|_H = \frac{\|\lambda^{k+1}-\lambda^k\|_H}{\rho_k}.
    \end{equation}
    Inserting this into \eqref{Eq:LemSigmaConv1} and using the triangle inequality yields
    \begin{equation*}
        \sigma_{k+1}\le \varepsilon_{k+1}+\frac{1}{\rho_k}
        \bigl(\|\lambda^{k+1}-\bar{\lambda}\|_H+\|\lambda^k-\bar{\lambda}\|_H\bigr).
    \end{equation*}
    Now, by Assumption \ref{Asm:LocalConv} and since $x^k\to\bar{x}$, there is a $c_1>0$ such that $\|\lambda^k-\bar{\lambda}\|_H\le c_1 \sigma_k$ for all $k\in\N$ sufficiently large. Hence,
    \begin{equation*}
        \sigma_{k+1}\le \varepsilon_{k+1}+\frac{c_1}{\rho_k}\sigma_{k+1}
        +\frac{c_1}{\rho_k}\sigma_k,
    \end{equation*}
    again for $k\in\N$ sufficiently large. Reordering gives the desired result.
\end{proof}

\noindent
With the above lemma, it is easy to deduce convergence rates for the primal-dual sequence $\{(x^k,\lambda^k)\}$.

\begin{thm}\label{Thm:ConvRate}
    Let Assumptions \ref{Asm:Subproblems}, \ref{Asm:LocalConv} hold, and let $\varepsilon_{k+1}=o(\sigma_k)$. Then:
    \begin{enumerate}[label=\textnormal{(\alph*)}]
        \item For every $q\in(0,1)$, there is a $\bar{\rho}_q>0$ such that, if $\rho_k\ge \bar{\rho}_q$ for sufficiently large $k$, then $(x^k,\lambda^k)\to (\bar{x},\bar{\lambda})$ Q-linearly with rate $q$.
        \item The sequence of penalty parameters $\{\rho_k\}$ remains bounded.
    \end{enumerate}
\end{thm}
\begin{proof}
    Let $k\in\N$ be sufficiently large so that $\BddMul^k=\lambda^k$. By Lemma \ref{Lem:SigmaConv}, if $\rho_k$ is large enough so that $1-c_1/\rho_k>0$, then
    \begin{equation}\label{Eq:ThmConvRate1}
        \frac{\sigma_{k+1}}{\sigma_k}\le \frac{c_1}{\rho_k-c_1}+o(1).
    \end{equation}
    Using \eqref{Eq:ErrorBound} and the local Lipschitz-continuity of $\sigma$ (e.g.\ equation \eqref{Eq:DoubleErrorBound}), it is easy to derive (a). For (b), let us again consider the sequence $s^{k+1}=P_K(g(x^{k+1})+\lambda^k/\rho_k)$, and define $V_{k+1}:=V(x^{k+1},\BddMul^k,\rho_k)=\|\Lag_{\rho_k}(x^{k+1},\BddMul^k)\|_{X^*}+\|g(x^{k+1})-s^{k+1}\|_H$. To prove boundedness of $\{\rho_k\}$, we need to show that $V_{k+1}\le\tau V_k$ for sufficiently large $k$. Using \eqref{Eq:LemSigmaConv2} and $\Lag_{\rho_k}(x^{k+1},\BddMul^k)=\Lag(x^{k+1},\lambda^{k+1})$, we obtain
    \begin{equation*}
        V_{k+1}\ge \|\Lag_{\rho_k}(x^{k+1},\BddMul^k)\|_{X^*} +
        \|g(x^{k+1})-P_K(g(x^{k+1})+\lambda^{k+1})\|_H=\sigma_{k+1}
    \end{equation*}
    for all $k\in\N$ and, from \eqref{Eq:LemSigmaConv2} and Assumption \ref{Asm:Subproblems},
    \begin{align*}
        V_{k+1} =\|\Lag_{\rho_k}(x^{k+1},\BddMul^k)\|_{X^*}+
        \frac{\|\lambda^{k+1}-\lambda^k\|_H}{\rho_k} 
        & \le \varepsilon_{k+1} + \frac{\|\lambda^{k+1}-\bar{\lambda}\|_H+
            \|\lambda^k-\bar{\lambda}\|_H}{\rho_k} \\
        & \le \varepsilon_{k+1} + \frac{c}{\rho_k}(\sigma_{k+1}+\sigma_k)
    \end{align*}
    for all $k\in\N$ sufficiently large, where $c$ is the constant from \eqref{Eq:ErrorBound} (recall that $x^k\to\bar{x}$). Putting these inequalities together yields
    \begin{equation*}
        \frac{V_{k+1}}{V_k}\le \frac{\varepsilon_{k+1}}{\sigma_k}+\frac{c}{\rho_k}
        \frac{\sigma_{k+1}+\sigma_k}{\sigma_k}=\frac{\varepsilon_{k+1}}{\sigma_k}
        +\frac{c}{\rho_k}\mleft( 1+\frac{\sigma_{k+1}}{\sigma_k} \mright).
    \end{equation*}
    If we now assume that $\rho_k\to\infty$, then it is easy to deduce from \eqref{Eq:ThmConvRate1} and $\varepsilon_{k+1}=o(\sigma_k)$ that $V_{k+1}/V_k\to 0$. Hence, $V_{k+1}/V_k\le\tau$ for all $k$ sufficiently large, which contradicts the assumption that $\rho_k\to\infty$.
\end{proof}

\noindent
The assumption $\varepsilon_{k+1}=o(\sigma_k)$ in the above theorem says that, roughly speaking, the degree of inexactness should be small enough to not affect the rate of convergence. Note that we are comparing $\varepsilon_{k+1}$ to the optimality measure $\sigma_k$ of the previous iterates $(x^k,\lambda^k)$. Hence, it is easy to ensure this condition in practice, for instance, by always computing the next iterate $x^{k+1}$ with a precision $\varepsilon_{k+1}\le z_k \sigma_k$ for some fixed null sequence $z_k$.

Let us also note that one can easily adapt the proof of Theorem \ref{Thm:ConvRate}(a) to conclude that $(x^k,\lambda^k)\to (\bar{x},\bar{\lambda})$ Q-superlinearly if $\rho_k\to\infty$. However, the resulting assertion would be redundant because part (b) of the theorem actually implies the boundedness of $\{\rho_k\}$. On the other hand, the proof of (b) uses the specific penalty updating scheme \eqref{Eq:RhoTest} with the function $V$ from \eqref{Eq:V}, whereas the proof of (a) does not depend on the penalty updating rule at all. If we replace $V$ by the function
\begin{equation*}
    \tilde{V}(x,\lambda,\rho):=\left\|g(x)-P_K\mleft(g(x)+
    \frac{\lambda}{\rho}\mright)\right\|_H
\end{equation*}
(which is just the second term from the definition of $V$), it is rather easy to see that the assertions of Lemmas \ref{Lem:ConvSOSC}, \ref{Lem:SigmaConv} and Theorem \ref{Thm:ConvRate}(a) remain true. Additionally, we obtain superlinear convergence if $\rho_k\to\infty$, but we do not get boundedness of $\{\rho_k\}$.

Let us close this section by mentioning two special cases for which different or stronger rate of convergence results can be obtained. The first case is that of convex optimization. In this case, the augmented Lagrangian algorithm is essentially equivalent to a proximal-point method (applied to the dual problem), and this duality can be used to establish certain rate of convergence results, see \cite{Dong2015,Guler1991,Kanzow2017b,Rockafellar1976}.

The second special case, which was already mentioned in the introduction, is that of nonlinear programming-type (NLP) constraints. Here, it is possible to prove local linear convergence under SOSC only \cite{Fernandez2012}. Constraint qualifications are not needed since the set $K$ is polyhedral, see the discussion in the introduction and in \cite[Section~4.4]{Bonnans2000}. However, the techniques used in \cite{Fernandez2012} rely heavily on finite-dimensional arguments and the specific structure of NLP constraints, and thus cannot readily be adapted to our setting.

\section{Applications}\label{Sec:Applic}

This section describes some applications of our method. Recall that our variational setting encompasses constrained optimization problems \eqref{Eq:Opt}. This opens up a broad spectrum of applications, including, as mentioned before, standard nonlinear programming (NLP). However, there already is a plethora of literature on this topic, in particular the recent paper \cite{Fernandez2012}. Moreover, the discussion in Section \ref{Sec:ErrorBounds} indicates that NLP is actually a very confined special case which does not allow us to demonstrate the full generality of our approach. In particular, NLPs are inherently finite-dimensional and the corresponding set $K$ is polyhedral, which is very restrictive.

As a result, we focus on problems in function space settings where the constraint set is almost never polyhedral. This section contains two examples in this direction: we begin with a simple linear-quadratic optimal control problem and then continue with multiobjective optimal control in a Nash equilibrium framework. For both examples, we first present the general problem setting and then explain why the regularity properties from Assumption~\ref{Asm:LocalConv} are satisfied.

To verify our theoretical results in practice, we follow a standard approach by which we discretize the respective problems and then analyze the behavior of the algorithm for increasingly fine levels of discretization. As we shall see, the assertions of the previous section can be verified in both examples, and independently of the dimension $n$, which indicates that our results are valid.

\subsection{An Optimal Control Problem}\label{Sec:ApplicOptCont}

Let $\Omega\subseteq\R^d$, $d\in\{2,3\}$, be a bounded domain. The example presented in this section consists of minimizing
\begin{equation*}
    J(y,u):=\frac{1}{2}\|y-y_d\|_{L^2(\Omega)}^2+\frac{\alpha}{2}\|u\|_{L^2(\Omega)}^2
\end{equation*}
subject to $y\in H_0^1(\Omega)\cap C(\bar{\Omega})$ and $u\in L^2(\Omega)$ satisfying the partial differential equation (PDE) and pointwise control constraints
\begin{equation*}
    -\Delta y=u+f \quad\text{and}\quad u_a\le u\le u_b.
\end{equation*}
Here, $y_d,u_a,u_b\in L^2(\Omega)$ are problem-specific and $\alpha>0$ is a regularization parameter. It is well-known that, for every right-hand side $w\in L^2(\Omega)$, the Poisson equation $-\Delta y=w$ admits a uniquely determined weak solution $y=S w\in H_0^1(\Omega)\cap C(\bar{\Omega})$, and the resulting operator $S:L^2(\Omega)\to H_0^1(\Omega)\cap C(\bar{\Omega})$ is linear and compact \cite[Thm.\ 4.17]{Troeltzsch2010}. Writing $y_u:=S(u+f)$, we can now restate the objective function as
\begin{equation*}
    \bar{J}(u):=J(y_u,u)=\frac{1}{2}\|y_u-y_d\|_{L^2(\Omega)}^2
    +\frac{\alpha}{2}\|u\|_{L^2(\Omega)}^2.
\end{equation*}
This function together with the control constraints $u_a\le u\le u_b$ is typically called the \emph{reduced formulation} of the optimal control problem and directly fits into our variational framework by setting $X:=H:=L^2(\Omega)$, $F(u):=\bar{J}'(u)$, and
\begin{equation*}
    g(u):=u, \quad K:=\{ u\in X: u_a\le u\le u_b \}.
\end{equation*}
Since $\bar{J}$ is strongly convex and $g$ is just the identity mapping on $X=H$, it is easy to show that the above problem admits a unique primal-dual solution, and that both SOSC and SRC hold. Hence, by Theorem \ref{Thm:ErrorBound}, the KKT system is upper Lipschitz stable and the control problem admits a local error bound.

We now present a numerical example which is constructed in such a way that the optimal solution is known analytically. Let $\Omega:=(0,1)^2$ be the unit square and define $\alpha:=1$, $u_a:= -0.5$, $u_b:= 0.5$. Consider the functions
\begin{equation*}
    \bar{y}(x):=\sin (\pi x_1)\sin (\pi x_2), \quad
    \bar{p}(x):=\sin (2\pi x_1)\sin (2\pi x_2),
\end{equation*}
and set $y_d:=\bar{y}+\Delta \bar{p}$. Now, using $\bar{u}:=P_{[u_a,u_b]}(-\bar{p}/\alpha)$ and $f:=-\Delta \bar{y}-\bar{u}$, it is easy to see that $\bar{u}$ is a solution to the problem. Moreover, $\bar{y}$ is the corresponding state, $\bar{p}$ the so-called adjoint state \cite{Troeltzsch2010}, and the Lagrange multiplier is given by $\bar{\lambda}:=-\bar{p}-\alpha \bar{u}$.

For the numerical testing, we discretized the problem by means of a uniform grid with $n\in\N$ interior points per row or column (i.e., $n^2$ points in total) and approximated the Laplace operator by a standard five-point finite difference scheme. It is easy to argue that the resulting discretized versions of $\bar{J}$ and $g$ again satisfy the (now finite-dimensional) SOSC and SRC assumptions (since $\bar{J}$ is strongly convex and $g$ is the identity mapping). Hence, we can expect locally fast convergence of the augmented Lagrangian method, both from a continuous and a discrete point of view.

The implementation of the algorithm was done in MATLAB\textsuperscript{\textregistered} and uses the parameters
\begin{equation*}
    (u^0,\lambda^0):=(0,0), \quad B:=[-10^6,10^6]^{n^2}, \quad
    \rho_0:=1, \quad \gamma:=10, \quad \tau:=0.5,
\end{equation*}
together with the formula $\BddMul^k:=P_B(\lambda^k)$ for the safeguarded multipliers (see the discussion in Section \ref{Sec:Method}). Moreover, we use the termination criteria $\sigma(x,\lambda)\le 10^{-8}$ and $\|\Lag_{\rho_k}(x,\BddMul^k)\|\le 10^{-10}$ for the outer and inner iterations, respectively, where the norm is the discrete $L^2$-norm. The subproblems are nonlinear equations which we solve with a standard semismooth Newton method. It should be noted that, while the discrete Laplacian is a sparse matrix, the solution operator $S$ which occurs in the function $F$ is nearly dense. To circumvent this issue, we use a sparse Cholesky factorization of the negative Laplacian to obtain an ``implicit'' form of $S$ and solve the Newton equations with the MATLAB\textsuperscript{\textregistered} conjugate gradient method \texttt{pcg}.

\begin{table}\centering
\begin{tabular}{|r|ccc|ccc|ccc|}\hline
    & \multicolumn{3}{|c|}{$n=64$} & \multicolumn{3}{|c|}{$n=256$} & \multicolumn{3}{|c|}{$n=1024$} \\
    $k$ & $\rho_k$ & $\sigma_k$ & $\operatorname{dist}_k$ &
    $\rho_k$ & $\sigma_k$ & $\operatorname{dist}_k$ &
    $\rho_k$ & $\sigma_k$ & $\operatorname{dist}_k$ \\ \hline

     0 &  1 & 5.08e-01 & 5.43e-01 &  1 & 5.02e-01 & 5.37e-01 &  1 & 5.01e-01 & 5.35e-01 \\
     1 &  1 & 8.58e-02 & 1.71e-01 &  1 & 8.47e-02 & 1.69e-01 &  1 & 8.44e-02 & 1.69e-01 \\
     2 &  1 & 4.29e-02 & 8.55e-02 &  1 & 4.23e-02 & 8.46e-02 &  1 & 4.22e-02 & 8.44e-02 \\
     3 & 10 & 2.15e-02 & 4.26e-02 & 10 & 2.12e-02 & 4.23e-02 & 10 & 2.11e-02 & 4.22e-02 \\
     4 & 10 & 1.95e-03 & 3.57e-03 & 10 & 1.92e-03 & 3.83e-03 & 10 & 1.92e-03 & 3.84e-03 \\
     5 & 10 & 1.77e-04 & 4.44e-04 & 10 & 1.75e-04 & 3.29e-04 & 10 & 1.74e-04 & 3.48e-04 \\
     6 & 10 & 1.61e-05 & 5.08e-04 & 10 & 1.59e-05 & 2.85e-05 & 10 & 1.59e-05 & 3.04e-05 \\
     7 & 10 & 1.47e-06 & 5.21e-04 & 10 & 1.45e-06 & 3.18e-05 & 10 & 1.44e-06 & 2.08e-06 \\
     8 & 10 & 1.33e-07 & 5.22e-04 & 10 & 1.31e-07 & 3.29e-05 & 10 & 1.31e-07 & 1.96e-06 \\
     9 & 10 & 1.21e-08 & 5.22e-04 & 10 & 1.20e-08 & 3.30e-05 & 10 & 1.19e-08 & 2.06e-06 \\
    10 & 10 & 1.10e-09 & 5.22e-04 & 10 & 1.09e-09 & 3.30e-05 & 10 & 1.08e-09 & 2.07e-06 \\

    \hline
\end{tabular}
\caption{Numerical results for the optimal control problem.}
\label{Tab:OptCont}
\end{table}

Table \ref{Tab:OptCont} lists some numerical results for different values of $n$, where each line contains the penalty parameter $\rho_k$, the optimality measure $\sigma_k$ and the distance $\operatorname{dist}_k$ of $(u^k,\lambda^k)$ to $(\bar{u},\bar{\lambda})$. The results suggest that the algorithm works very well for this problem; in particular, the number of required iterations remains constant as $n$ increases. Moreover, we also observe that the rate of convergence appears to be proportional to $1/\rho_k$, as suggested by the theory. It should be noted, however, that the distances $\operatorname{dist}_k$ stop decreasing after a certain point because of the inexactness induced by the discretization; in particular, if we discretize the (known) optimal solution pair $(\bar{u},\bar{\lambda})$, we do not obtain an \emph{exact} solution of the discretized problem. This phenomenon is also evidenced by the fact that the ``limit'' value of $\operatorname{dist}_k$ decreases as $n$ increases.

We close this section with an important remark on the analytical representation of the feasible set. This observation is crucial and was in fact one of our main motivations to consider constraint sets $K$ which are not necessarily cones.

\begin{rem}\label{Rem:BoxConstraints}
It is important that we define the constraint system with $g$ and $K$ as above. Indeed, the alternative formulation of the box constraints as $\hat{g}(u)\in \hat{K}$ with
\begin{equation*}
    \hat{g}(u):=(u-u_a, u_b-u),
    \quad \hat{K}:=\{(v,w)\in L^2(\Omega)^2: v,w\ge 0 \},
\end{equation*}
may seem advantageous at first glance (since $\hat{K}$ is a closed convex cone, whereas $K$ is not). However, in this formulation, the strict Robinson condition is not satisfied. In fact, the function $\hat{g}$ does not even satisfy the standard Robinson constraint qualification (RCQ) \cite{Bonnans2000} or the equivalent regularity condition of Zowe and Kurcyusz \cite{Zowe1979}. We refer the reader to \cite{Troeltzsch2010} for a formal proof; an alternative way to verify this irregularity is to note that if RCQ holds, then it remains stable under small perturbations of the constraint function \cite{Bonnans2000}. However, even if $u_a$ and $u_b$ are ``well separated'', it is fairly easy to construct small perturbations (in the sense of $L^2$) which make the lower and upper bounds coincide on some set of positive measure. If this happens, then the set of Lagrange multipliers corresponding to a local minimum is unbounded, and RCQ is violated.
\end{rem}

\subsection{Optimal Control in a Nash Equilibrium Framework}

We now present a generalization of the optimal control problem from the previous section by considering it in a multi-player framework \cite{Borzi2013,Dreves2016,Kanzow2017a}. The result is a Nash equilibrium problem (NEP) of two players with control variables $u_1,u_2\in L^2(\Omega)$ and a state variable $y\in H_0^1(\Omega)\cap C(\bar{\Omega})$, where $\Omega\subseteq\R^d$, $d\in\{2,3\}$, is again a bounded domain. Similarly to before, each player attempts to minimize the objective function
\begin{equation*}
    J_i(y,u_i):=\frac{1}{2}\|y-y_d^i\|_{L^2(\Omega)}^2+
    \frac{\alpha_i}{2}\|u_i\|_{L^2(\Omega)}^2
\end{equation*}
with respect to $u_i$, subject to the partial differential equation $-\Delta y=u_1+u_2+f$ and the pointwise control constraints $a_i \le u_i \le b_i$ with $a_i,b_i\in L^2(\Omega)$. The remaining problem parameters satisfy $\alpha_i>0$ and $y_d^i\in L^2(\Omega)$ for all $i$. As in Section \ref{Sec:ApplicOptCont}, we can use the compact linear solution operator $S:L^2(\Omega)\to H_0^1(\Omega)\cap C(\bar{\Omega})$ and the resulting control-to-state mapping $y_u:=S(u_1+u_2+f)$ to transform the objective functions to
\begin{equation*}
    \bar{J}_i(u):=J_i(y_u,u_i)
    =\frac{1}{2}\|y_u-y_d^i\|_{L^2(\Omega)}^2
    +\frac{\alpha_i}{2}\|u_i\|_{L^2(\Omega)}^2,
\end{equation*}
where $u:=(u_1,u_2)$. To establish the connection with our variational problem \eqref{Eq:VI}, \eqref{Eq:M}, we only need to make some definitions and use the well-known correspondence between NEPs and VIs \cite{Facchinei2007,Kanzow2017a}. Define $X:=H:=L^2(\Omega)^2$, $F(u):=\bigl(D_{u_1} \bar{J}_1(u),D_{u_2} \bar{J}_2(u)\bigr)$, and
\begin{equation*}
    g(u_1,u_2):=(u_1,u_2), \quad
    K:=\{(u_1,u_2)\in X: a_i\le u_i\le b_i \}.
\end{equation*}
Then it is easy to see that the NEP is equivalent to the VI \eqref{Eq:VI} (and \eqref{Eq:VI_Convex}, since the feasible set is convex). The existence of a solution of the NEP (and of the VI) can be shown as in \cite{Borzi2013}; moreover, since $g$ is the identity operator on $X=H$, SRC holds and the problem admits a unique Lagrange multiplier. Finally, an easy calculation shows that
\begin{equation*}
    F'(u)=
    \begin{pmatrix}
        S^* S+\alpha_1 I & S^* S \\ S^* S & S^* S+\alpha_2 I
    \end{pmatrix},
\end{equation*}
where $I$ is the identity operator on $L^2(\Omega)$, see \cite{Kanzow2017a}. It follows that $F$ is strongly monotone and, since $g$ is linear, the problem automatically satisfies SOSC and therefore admits a local error bound by Theorem \ref{Thm:ErrorBound}. Moreover, it is easy to see that the same holds for the discretized problems presented below.

We now present some numerical results for the example from \cite{Borzi2013}. The setting is again constructed in such a way that the optimal solution is known. In fact, the construction is very similar to the one from the previous section: let $\Omega:=(0,1)^2$ be the unit square and define $\alpha_i:=1$, $a_i:=-0.5$, and $b_i:=0.5$ for all $i$. Consider the functions
\begin{equation*}
    \bar{y}(x) :=\sin (\pi x_1)\sin (\pi x_2), \quad
\begin{aligned}
    \bar{p}_1(x) & :=-\sin (2\pi x_1)\sin (2\pi x_2), \\
    \bar{p}_2(x) & :=-\sin (3\pi x_1)\sin (3\pi x_2),
\end{aligned}
\end{equation*}
as well as $y_d^i:=\bar{y}+\Delta \bar{p}_i$, $\bar{u}_i:=P_{[a_i,b_i]}(-\bar{p}_i/\alpha_i)$ for all $i$, and finally $f:=-\Delta \bar{y}-\bar{u}_1-\bar{u}_2$. Then it is easy to see that $\bar{u}$ is a Nash equilibrium. The corresponding state is given by $\bar{y}$, the variables $\bar{p}_i$ are the adjoint states of the players, and the Lagrange multiplier is given by $\bar{\lambda}:=(-\bar{p}_1-\alpha_1\bar{u}_1,-\bar{p}_2-\alpha_2\bar{u}_2)$.

\begin{table}\centering
\begin{tabular}{|r|ccc|ccc|ccc|}\hline
    & \multicolumn{3}{|c|}{$n=64$} & \multicolumn{3}{|c|}{$n=256$} & \multicolumn{3}{|c|}{$n=1024$} \\
    $k$ & $\rho_k$ & $\sigma_k$ & $\operatorname{dist}_k$ &
    $\rho_k$ & $\sigma_k$ & $\operatorname{dist}_k$ &
    $\rho_k$ & $\sigma_k$ & $\operatorname{dist}_k$ \\ \hline
    
     0 &  1 & 5.08e-01 & 5.43e-01 &  1 & 5.02e-01 & 5.37e-01 &  1 & 5.01e-01 & 5.35e-01 \\
     1 &  1 & 8.59e-02 & 1.71e-01 &  1 & 8.47e-02 & 1.69e-01 &  1 & 8.44e-02 & 1.69e-01 \\
     2 &  1 & 4.30e-02 & 8.54e-02 &  1 & 4.23e-02 & 8.46e-02 &  1 & 4.22e-02 & 8.44e-02 \\
     3 & 10 & 2.15e-02 & 4.24e-02 & 10 & 2.12e-02 & 4.23e-02 & 10 & 2.11e-02 & 4.22e-02 \\
     4 & 10 & 1.95e-03 & 3.41e-03 & 10 & 1.92e-03 & 3.81e-03 & 10 & 1.92e-03 & 3.83e-03 \\
     5 & 10 & 1.78e-04 & 8.13e-04 & 10 & 1.75e-04 & 3.17e-04 & 10 & 1.74e-04 & 3.47e-04 \\
     6 & 10 & 1.61e-05 & 8.95e-04 & 10 & 1.59e-05 & 5.08e-05 & 10 & 1.59e-05 & 2.96e-05 \\
     7 & 10 & 1.47e-06 & 9.08e-04 & 10 & 1.45e-06 & 5.63e-05 & 10 & 1.44e-06 & 3.23e-06 \\
     8 & 10 & 1.33e-07 & 9.09e-04 & 10 & 1.31e-07 & 5.74e-05 & 10 & 1.31e-07 & 3.50e-06 \\
     9 & 10 & 1.21e-08 & 9.09e-04 & 10 & 1.20e-08 & 5.75e-05 & 10 & 1.19e-08 & 3.59e-06 \\
    10 & 10 & 1.10e-09 & 9.09e-04 & 10 & 1.09e-09 & 5.75e-05 & 10 & 1.08e-09 & 3.60e-06 \\

    \hline
\end{tabular}
\caption{Numerical results for the optimal control Nash equilibrium problem.}
\label{Tab:OptContNEP}
\end{table}

The implementation of the augmented Lagrangian method for the above problem is similar to that of the previous section. More precisely, we use the same set of parameters, the same termination criteria, and the same method for the solution of the subproblems. The corresponding numerical results are given in Table \ref{Tab:OptContNEP}, where each line contains the values of the penalty parameter $\rho_k$, the optimality measure $\sigma_k$, and the distance $\operatorname{dist}_k$ of $(u^k,\lambda^k)$ to $(\bar{u},\bar{\lambda})$. We observe good consistency of the results with our established theory; in particular, the rate of convergence is roughly proportional to $1/\rho_k$. We also highlight once again that the distances $\operatorname{dist}_k$ do not converge to zero because of the inexactness induced by the discretization.

We close this section by noting that, as explained in Remark \ref{Rem:BoxConstraints} for the standard (single-objective) optimal control problem, it is very important that we define $g$ and $K$ precisely as we did in order to ensure the fulfillment of the strict Robinson condition.

\subsection{Parameter Estimation in Elliptic Systems}\label{Sec:ApplicPara}

This example is based on the theory in \cite{Ito1990b,Ito1991}. For the sake of simplicity, we restrict ourselves to the one-dimensional case. Let $\Omega\subseteq \R$ be a bounded interval and consider the elliptic differential equation
\begin{equation}\label{Eq:ParaEstPDE}
    -\nabla (q \nabla u )=f, \quad u\in H_0^1(\Omega),
\end{equation}
where $q\in H^1(\Omega)$ and $f\in H^{-1}(\Omega)$. The parameter estimation problem now consists of the minimization of the tracking-type functional
\begin{equation}\label{Eq:ParaEstOpt}
    J(q,u):= \frac{1}{2}\|u-z\|_{H_0^1(\Omega)}^2+\frac{\beta}{2}\|q\|_{H^1(\Omega)}^2
\end{equation}
subject to \eqref{Eq:ParaEstPDE} and $q\ge \alpha$, where $z\in H_0^1(\Omega)$ and $\alpha,\beta>0$. To formulate this problem in our variational framework, let $X:=H:=H^1(\Omega)\times H_0^1(\Omega)$, $F:=( D_q J, D_u J )$, and
\begin{equation*}
    g(q,u):=
    \begin{pmatrix}
        q-\alpha \\ -\Delta^{-1}\bigl( \nabla (q \nabla u)+f \bigr)
    \end{pmatrix}
    , \quad K:=H_+^1(\Omega) \times \{0\},
\end{equation*}
where $H_+^1(\Omega)$ is the nonnegative cone in $H^1(\Omega)$. Note that the second component of $g$ is essentially the differential equation \eqref{Eq:ParaEstPDE}, but premultiplied with $-\Delta^{-1}$ to map the result back into $H_0^1(\Omega)$.

\begin{figure}\centering
\includegraphics[scale=0.5]{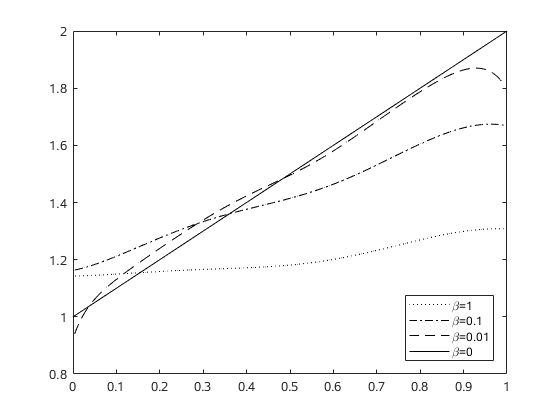}
\includegraphics[scale=0.5]{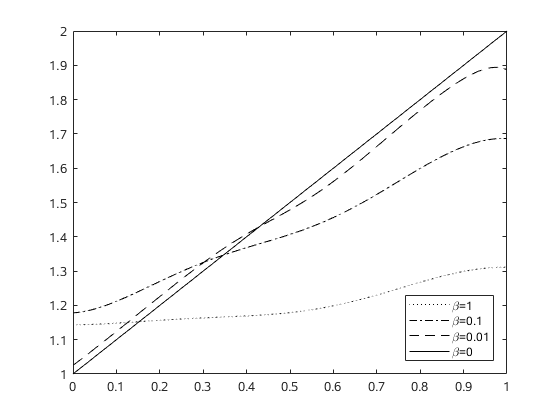}
\caption{Computed solutions $q$ of the parameter estimation problem for $n=256$ (left) and $n=1024$ (right).}
\label{Fig:ParaEst}
\end{figure}

The existence of solutions to \eqref{Eq:ParaEstOpt} can be shown by eliminating $u$ in \eqref{Eq:ParaEstPDE} and using the coercivity of $J$, see \cite{Ito1990b}. Let $(\bar{q},\bar{u})$ be a solution of the problem. Then
\begin{equation*}
    g'(\bar{q},\bar{u})=
    \begin{pmatrix}
        \operatorname{id}_{H^1} & 0 \\
        T_{\bar{u}} & T_{\bar{q}}
    \end{pmatrix},
\end{equation*}
where $T_{\bar{u}}(q):=-\Delta^{-1}( \nabla (q\nabla \bar{u}) )$ and $T_{\bar{q}}(u):=-\Delta^{-1}(\nabla (\bar{q}\nabla u))$. Observe now that $T_{\bar{q}}:H_0^1(\Omega)\to H_0^1(\Omega)$ is surjective. This follows from the fact that $\Delta:H_0^1(\Omega)\to H^{-1}(\Omega)$ is an isomorphism and that $u\mapsto \nabla (\bar{q}\nabla u)$ is surjective onto $H^{-1}(\Omega)$ by the Lax--Milgram theorem (since $\bar{q}\ge \alpha>0$), see \cite{Troeltzsch2010}. It therefore follows that the whole operator $g'(\bar{q},\bar{u})$ is surjective, and thus the strict Robinson condition is satisfied in $(\bar{q},\bar{u})$.

Let us furthermore assume that the second-order sufficient condition holds in $(\bar{q},\bar{u})$. The precise verification of this condition would require the knowledge of the solution, but the second-order condition is very plausible since the objective in \eqref{Eq:ParaEstOpt} is strongly convex (by virtue of the $H^1$-regularization term). Under the present assumptions, the problem admits the local error bound from Theorem~\ref{Thm:ErrorBound}. The corresponding residual mapping $\sigma:X\times H\to \R$ takes on the form
\begin{equation*}
    \sigma(q,u,\mu,\lambda):=\|F(q,u)+g'(q,u)^* (\mu,\lambda)\|_{X^*}+
    \|g(q,u)-P_K (g(q,u)+(\mu,\lambda)) \|_H,
\end{equation*}
where $(\mu,\lambda)\in H=H^1(\Omega)\times H_0^1(\Omega)$ is the pair of Lagrange multipliers. The vector $\mu$ corresponds to the lower bound constraint $q\ge \alpha$ (the first component of $g$), whereas $\lambda$ belongs to the partial differential equation (the second component of $g$).

\begin{table}\centering
\begin{tabular}{|r|cc|cc|cc|cc|}\hline
    & \multicolumn{2}{|c|}{$n=256,\,\beta=1$} & \multicolumn{2}{|c|}{$n=256,\,\beta=0.01$}
    & \multicolumn{2}{|c|}{$n=1024,\,\beta=1$} & \multicolumn{2}{|c|}{$n=1024,\,\beta=0.01$} \\
    $k$ & $\rho_k$ & $\sigma_k$ & $\rho_k$ & $\sigma_k$ & $\rho_k$ & $\sigma_k$
    & $\rho_k$ & $\sigma_k$ \\ \hline
    
     0 &  1 & 2.54e+04 &  1 & 2.54e+04 & 1 & 2.05e+05 & 1 & 2.05e+05 \\
     1 &  1 & 4.64e-01 &  1 & 1.21e-01 & 1 & 4.44e-01 & 1 & 6.59e-02 \\
     2 &  1 & 7.48e-02 &  1 & 5.01e-02 & 1 & 5.83e-02 & 1 & 2.52e-02 \\
     3 &  1 & 2.75e-02 &  1 & 2.50e-02 & 1 & 1.94e-02 & 1 & 1.17e-02 \\
     4 &  1 & 1.12e-02 & 10 & 1.30e-02 & 1 & 7.54e-03 & 1 & 5.64e-03 \\
     5 &  1 & 4.62e-03 & 10 & 1.30e-03 & 1 & 3.00e-03 & 1 & 2.73e-03 \\
     6 &  1 & 1.93e-03 & 10 & 1.30e-04 & 1 & 1.21e-03 & 1 & 1.33e-03 \\
     7 &  1 & 8.16e-04 & 10 & 1.31e-05 & 1 & 4.97e-04 & 1 & 6.46e-04 \\
     8 &  1 & 3.46e-04 &    &          & 1 & 4.80e-03 & 1 & 3.14e-04 \\
     9 &  1 & 1.48e-04 &    &          & 1 & 8.67e-05 & 1 & 1.53e-04 \\
    10 &  1 & 6.34e-05 &    &          &   &          & 1 & 7.48e-05 \\
    
    \hline
\end{tabular}
\caption{Iteration histories for the parameter estimation problem.}
\label{Tab:ParaEst}
\end{table}

We now present some numerical results. For practical purposes, we slightly alter the penalization scheme from Algorithm~\ref{Alg:ALM} in the sense that we augment the equality constraint only and leave the inequality constraint $q\ge \alpha$ unchanged. This has the benefit that we avoid the computation of projections and distance functions involving $H_+^1(\Omega)$. The resulting modifications to Algorithm~\ref{Alg:ALM} are fairly straightforward (see, for instance, \cite{Birgin2012,Birgin2014,Kanzow2017a}). Indeed, the augmented subproblems are now (constrained) variational inequalities over the set $\{q\in H^1(\Omega): q\ge \alpha \}$. Moreover, in the updating scheme of the penalty parameter, we have to take into account the multiplier corresponding to the lower inequality constraint, which has to be recovered from the solution process of the corresponding constrained subproblem.

The example we present is \cite[Ex.~6]{Ito1991}. The domain $\Omega:=(0,1)$ is discretized by means of $n\in\N$ points, including boundary points, and the derivative operators are approximated by forward differences. The problem is constructed by setting
\begin{equation*}
    q_0(x):=1+x, \quad z(x):=u_0(x):=\sin (\pi x), \quad
    f(x):=(1+x)\pi^2 \sin (\pi x)-\pi \cos (\pi x),
\end{equation*}
so that $-\nabla (q_0 \nabla u_0)=f$. Since $z=u_0$, an exact solution of \eqref{Eq:ParaEstOpt} for $\beta=0$ is simply given by $(q_0,u_0)$. For $\beta>0$, which is the preferable case from a numerical perspective, the solutions are different in general.

The implementation of the algorithm was done in MATLAB\textsuperscript{\textregistered} and uses the parameters
\begin{equation*}
    (q^0,u^0,\mu^0,\lambda^0):=(1,0,0,0), \quad \alpha:=0.1, \quad
    \rho_0:=1, \quad \gamma:=10, \quad \tau:=0.5,
\end{equation*}
together with $\BddMul^k:=P_B(\lambda^k)$ and $B$ the closed ball with radius $10^6$ around zero in $H_0^1(\Omega)$. The termination criteria for the outer and inner iterations are $\sigma(q,u,\mu,\lambda)\le 10^{-4}$ and $\|\Lag_{\rho_k}(q,u,\BddMul^k)+\mu^k\|_{X^*}\le 10^{-6}$, respectively, where $\mu^k$ is the Lagrange multiplier corresponding to the constraint $q\ge \alpha$. Finally, the augmented subproblems were solved by the \texttt{fmincon} routine which takes into account the lower box constraint.

Table~\ref{Tab:ParaEst} contains the corresponding iteration numbers for different values of $n$ and $\beta$. We again observe linear convergence of the optimality measures $\sigma_k$, and the sequences of penalty parameters remain bounded. The only exception is the eighth iteration for $n=1024$ and $\beta=1$, which may be due to the subproblem routine \texttt{fmincon} failing to find a sufficiently exact minimizer. Finally, Figure~\ref{Fig:ParaEst} compares the computed solutions $q$ for different $n$ and $\beta$ to the exact solution $q_0$ for $\beta=0$.

\section{Final Remarks}\label{Sec:Final}

We have presented a method of augmented Lagrangian type for the solution of variational problems in Banach spaces. In particular, we have shown global and local convergence of the algorithm under suitable assumptions.

The assumptions needed for the local convergence results include, in particular, a local error bound for the distance of a pair $(x,\lambda)$ to a KKT point $(\bar{x},\bar{\lambda})$. This property has played a central role in our analysis and is a consequence of the second-order sufficient condition together with a strict version of the Robinson constraint qualification.

The above results suggest that error bounds are the natural framework for the local convergence analysis of augmented Lagrangian methods. We therefore hope that the results in this paper will find applications in other areas of optimization. In particular, an interesting idea would be to specialize some of the assumptions and results for problem classes such as optimal control or semidefinite programming. Another aspect which could lead to further developments is the concept of \emph{partial penalization} which arises when additional constraints are present in the problem formulation which are not penalized, see \cite{Andreani2007,Birgin2012,Birgin2014} and the example in Section~\ref{Sec:ApplicPara}.

\bibliographystyle{abbrv}
\bibliography{VI_ALMinf}

\begin{thebibliography}{10}

\bibitem{Andreani2007}
R.~Andreani, E.~G. Birgin, J.~M. Mart\'inez, and M.~L. Schuverdt.
\newblock On augmented {L}agrangian methods with general lower-level
  constraints.
\newblock {\em SIAM J. Optim.}, 18(4):1286--1309, 2007.

\bibitem{Andreani2008}
R.~Andreani, E.~G. Birgin, J.~M. Mart\'inez, and M.~L. Schuverdt.
\newblock Augmented {L}agrangian methods under the constant positive linear
  dependence constraint qualification.
\newblock {\em Math. Program.}, 111(1-2, Ser. B):5--32, 2008.

\bibitem{Baiocchi1984}
C.~Baiocchi and A.~Capelo.
\newblock {\em Variational and {Q}uasivariational {I}nequalities}.
\newblock John Wiley \& Sons, Inc., New York, 1984.

\bibitem{Bauschke2011}
H.~H. Bauschke and P.~L. Combettes.
\newblock {\em Convex {A}nalysis and {M}onotone {O}perator {T}heory in
  {H}ilbert {S}paces}.
\newblock Springer, New York, 2011.

\bibitem{Bertsekas1982}
D.~P. Bertsekas.
\newblock {\em Constrained {O}ptimization and {L}agrange {M}ultiplier
  {M}ethods}.
\newblock Academic Press, Inc. [Harcourt Brace Jovanovich, Publishers], New
  York-London, 1982.

\bibitem{Birgin2012}
E.~G. Birgin, D.~Fern\'andez, and J.~M. Mart\'inez.
\newblock The boundedness of penalty parameters in an augmented {L}agrangian
  method with constrained subproblems.
\newblock {\em Optim. Methods Softw.}, 27(6):1001--1024, 2012.

\bibitem{Birgin2010}
E.~G. Birgin, C.~A. Floudas, and J.~M. Mart\'inez.
\newblock Global minimization using an augmented {L}agrangian method with
  variable lower-level constraints.
\newblock {\em Math. Program.}, 125(1, Ser. A):139--162, 2010.

\bibitem{Birgin2014}
E.~G. Birgin and J.~M. Mart{\'{\i}}nez.
\newblock {\em Practical {A}ugmented {L}agrangian {M}ethods for {C}onstrained
  {O}ptimization}.
\newblock Society for Industrial and Applied Mathematics (SIAM), Philadelphia,
  PA, 2014.

\bibitem{Bonnans2000}
J.~F. Bonnans and A.~Shapiro.
\newblock {\em Perturbation {A}nalysis of {O}ptimization {P}roblems}.
\newblock Springer Series in Operations Research. Springer-Verlag, New York,
  2000.

\bibitem{Borzi2013}
A.~Borz{\`\i} and C.~Kanzow.
\newblock Formulation and numerical solution of {N}ash equilibrium
  multiobjective elliptic control problems.
\newblock {\em SIAM J. Control Optim.}, 51(1):718--744, 2013.

\bibitem{Conn1991}
A.~R. Conn, N.~I.~M. Gould, and P.~L. Toint.
\newblock A globally convergent augmented {L}agrangian algorithm for
  optimization with general constraints and simple bounds.
\newblock {\em SIAM J. Numer. Anal.}, 28(2):545--572, 1991.

\bibitem{Ding2017}
C.~Ding, D.~Sun, and L.~Zhang.
\newblock Characterization of the robust isolated calmness for a class of conic
  programming problems.
\newblock {\em SIAM J. Optim.}, 27(1):67--90, 2017.

\bibitem{Dong2015}
Y.~Dong.
\newblock Comments on ``{T}he proximal point algorithm revisited''.
\newblock {\em J. Optim. Theory Appl.}, 166(1):343--349, 2015.

\bibitem{Dontchev1998}
A.~L. Dontchev and R.~T. Rockafellar.
\newblock Characterizations of {L}ipschitzian stability in nonlinear
  programming.
\newblock In {\em Mathematical programming with data perturbations}, volume 195
  of {\em Lecture Notes in Pure and Appl. Math.}, pages 65--82. Dekker, New
  York, 1998.

\bibitem{Dreves2016}
A.~Dreves and J.~Gwinner.
\newblock Jointly convex generalized {N}ash equilibria and elliptic
  multiobjective optimal control.
\newblock {\em J. Optim. Theory Appl.}, 168(3):1065--1086, 2016.

\bibitem{Facchinei2007}
F.~Facchinei, A.~Fischer, and V.~Piccialli.
\newblock On generalized {N}ash games and variational inequalities.
\newblock {\em Oper. Res. Lett.}, 35(2):159--164, 2007.

\bibitem{Facchinei2010}
F.~Facchinei and C.~Kanzow.
\newblock Generalized {N}ash equilibrium problems.
\newblock {\em Ann. Oper. Res.}, 175:177--211, 2010.

\bibitem{Facchinei2003}
F.~Facchinei and J.-S. Pang.
\newblock {\em Finite-{D}imensional {V}ariational {I}nequalities and
  {C}omplementarity {P}roblems. {V}ol. {I}}.
\newblock Springer-Verlag, New York, 2003.

\bibitem{Fernandez2012}
D.~Fern{\'a}ndez and M.~V. Solodov.
\newblock Local convergence of exact and inexact augmented {L}agrangian methods
  under the second-order sufficient optimality condition.
\newblock {\em SIAM J. Optim.}, 22(2):384--407, 2012.

\bibitem{Fischer2002}
A.~Fischer.
\newblock Local behavior of an iterative framework for generalized equations
  with nonisolated solutions.
\newblock {\em Math. Program.}, 94(1, Ser. A):91--124, 2002.

\bibitem{Fischer2014}
A.~Fischer, M.~Herrich, and K.~Sch\"onefeld.
\newblock Generalized {N}ash equilibrium problems - recent advances and
  challenges.
\newblock {\em Pesquisa Operacional}, 34:521 -- 558, 12 2014.

\bibitem{Floudas2001}
C.~A. Floudas and P.~M. Pardalos, editors.
\newblock {\em Encyclopedia of Optimization. {V}ol. {I}--{VI}}.
\newblock Kluwer Academic Publishers, Dordrecht, 2001.

\bibitem{Fortin1983}
M.~Fortin and R.~Glowinski.
\newblock {\em Augmented {L}agrangian Methods: Applications to the Numerical
  Solution of Boundary-Value Problems}, volume~15 of {\em Studies in
  Mathematics and its Applications}.
\newblock North-Holland Publishing Co., Amsterdam, 1983.

\bibitem{Glowinski2008}
R.~Glowinski.
\newblock {\em Numerical Methods for Nonlinear Variational Problems}.
\newblock Scientific Computation. Springer-Verlag, Berlin, 2008.
\newblock Reprint of the 1984 original.

\bibitem{Glowinski2015}
R.~Glowinski.
\newblock {\em Variational Methods for the Numerical Solution of Nonlinear
  Elliptic Problems}, volume~86 of {\em CBMS-NSF Regional Conference Series in
  Applied Mathematics}.
\newblock Society for Industrial and Applied Mathematics (SIAM), Philadelphia,
  PA, 2015.

\bibitem{Glowinski1981}
R.~Glowinski, J.-L. Lions, and R.~Tr\'emoli\`eres.
\newblock {\em Numerical Analysis of Variational Inequalities}, volume~8 of
  {\em Studies in Mathematics and its Applications}.
\newblock North-Holland Publishing Co., Amsterdam-New York, 1981.

\bibitem{Guler1991}
O.~G\"uler.
\newblock On the convergence of the proximal point algorithm for convex
  minimization.
\newblock {\em SIAM J. Control Optim.}, 29(2):403--419, 1991.

\bibitem{Hestenes1969}
M.~R. Hestenes.
\newblock Multiplier and gradient methods.
\newblock {\em J. Optimization Theory Appl.}, 4:303--320, 1969.

\bibitem{Hintermueller2006}
M.~Hinterm{\"u}ller and K.~Kunisch.
\newblock Feasible and noninterior path-following in constrained minimization
  with low multiplier regularity.
\newblock {\em SIAM J. Control Optim.}, 45(4):1198--1221, 2006.

\bibitem{Hintermueller2015}
M.~Hinterm\"uller, T.~Surowiec, and A.~K\"ammler.
\newblock Generalized {N}ash equilibrium problems in {B}anach spaces: theory,
  {N}ikaido-{I}soda-based path-following methods, and applications.
\newblock {\em SIAM J. Optim.}, 25(3):1826--1856, 2015.

\bibitem{Isac1992}
G.~Isac.
\newblock {\em Complementarity Problems}, volume 1528 of {\em Lecture Notes in
  Mathematics}.
\newblock Springer-Verlag, Berlin, 1992.

\bibitem{Ito1991}
K.~Ito, M.~Kroller, and K.~Kunisch.
\newblock A numerical study of an augmented {L}agrangian method for the
  estimation of parameters in elliptic systems.
\newblock {\em SIAM J. Sci. Statist. Comput.}, 12(4):884--910, 1991.

\bibitem{Ito1990a}
K.~Ito and K.~Kunisch.
\newblock The augmented {L}agrangian method for equality and inequality
  constraints in {H}ilbert spaces.
\newblock {\em Math. Programming}, 46(3, (Ser. A)):341--360, 1990.

\bibitem{Ito1990b}
K.~Ito and K.~Kunisch.
\newblock The augmented {L}agrangian method for parameter estimation in
  elliptic systems.
\newblock {\em SIAM J. Control Optim.}, 28(1):113--136, 1990.

\bibitem{Ito2000}
K.~Ito and K.~Kunisch.
\newblock Augmented {L}agrangian methods for nonsmooth, convex optimization in
  {H}ilbert spaces.
\newblock {\em Nonlinear Anal.}, 41(5-6, Ser. A: Theory Methods):591--616,
  2000.

\bibitem{Ito2008}
K.~Ito and K.~Kunisch.
\newblock {\em Lagrange {M}ultiplier {A}pproach to {V}ariational {P}roblems and
  {A}pplications}.
\newblock Society for Industrial and Applied Mathematics (SIAM), Philadelphia,
  PA, 2008.

\bibitem{Izmailov2012a}
A.~F. Izmailov and M.~V. Solodov.
\newblock Stabilized {SQP} revisited.
\newblock {\em Math. Program.}, 133(1-2, Ser. A):93--120, 2012.

\bibitem{Kanzow2017a}
C.~Kanzow, V.~Karl, D.~Steck, and D.~Wachsmuth.
\newblock The multiplier-penalty method for generalized {N}ash equilibrium
  problems in {B}anach spaces.
\newblock {\em Technical Report}, Institute of Mathematics, University of
  W\"urzburg, July 2017.

\bibitem{Kanzow2017}
C.~Kanzow and D.~Steck.
\newblock An example comparing the standard and safeguarded augmented
  {L}agrangian methods.
\newblock {\em Oper. Res. Lett.}, 45(6):598--603, 2017.

\bibitem{Kanzow2017b}
C.~Kanzow and D.~Steck.
\newblock A generalized proximal-point method for convex optimization problems
  in {H}ilbert spaces.
\newblock {\em Optimization}, 66(10):1667--1676, 2017.

\bibitem{Kanzow2016}
C.~Kanzow, D.~Steck, and D.~Wachsmuth.
\newblock An augmented {L}agrangian method for optimization problems in
  {B}anach spaces.
\newblock {\em SIAM J. Control Optim.}, to appear.

\bibitem{Kinderlehrer2000}
D.~Kinderlehrer and G.~Stampacchia.
\newblock {\em An Introduction to Variational Inequalities and Their
  Applications}, volume~31 of {\em Classics in Applied Mathematics}.
\newblock Society for Industrial and Applied Mathematics (SIAM), Philadelphia,
  PA, 2000.
\newblock Reprint of the 1980 original.

\bibitem{Kyparisis1985}
J.~Kyparisis.
\newblock On uniqueness of {K}uhn-{T}ucker multipliers in nonlinear
  programming.
\newblock {\em Math. Programming}, 32(2):242--246, 1985.

\bibitem{Moreau1962}
J.-J. Moreau.
\newblock D\'ecomposition orthogonale d'un espace hilbertien selon deux c\^ones
  mutuellement polaires.
\newblock {\em C. R. Acad. Sci. Paris}, 255:238--240, 1962.

\bibitem{Nocedal2006}
J.~Nocedal and S.~J. Wright.
\newblock {\em Numerical {O}ptimization}.
\newblock Springer, New York, second edition, 2006.

\bibitem{Pang2005}
J.-S. Pang and M.~Fukushima.
\newblock Quasi-variational inequalities, generalized {N}ash equilibria, and
  multi-leader-follower games.
\newblock {\em Comput. Manag. Sci.}, 2(1):21--56, 2005.

\bibitem{Powell1969}
M.~J.~D. Powell.
\newblock A method for nonlinear constraints in minimization problems.
\newblock In {\em Optimization ({S}ympos., {U}niv. {K}eele, {K}eele, 1968)},
  pages 283--298. Academic Press, London, 1969.

\bibitem{Rockafellar1973}
R.~T. Rockafellar.
\newblock A dual approach to solving nonlinear programming problems by
  unconstrained optimization.
\newblock {\em Math. Programming}, 5:354--373, 1973.

\bibitem{Rockafellar1974}
R.~T. Rockafellar.
\newblock Augmented {L}agrange multiplier functions and duality in nonconvex
  programming.
\newblock {\em SIAM J. Control}, 12:268--285, 1974.

\bibitem{Rockafellar1976}
R.~T. Rockafellar.
\newblock Augmented {L}agrangians and applications of the proximal point
  algorithm in convex programming.
\newblock {\em Math. Oper. Res.}, 1(2):97--116, 1976.

\bibitem{Troeltzsch2010}
F.~Tr{\"o}ltzsch.
\newblock {\em Optimal {C}ontrol of {P}artial {D}ifferential {E}quations}.
\newblock American Mathematical Society, Providence, RI, 2010.

\bibitem{Ulbrich2011}
M.~Ulbrich.
\newblock {\em Semismooth {N}ewton Methods for Variational Inequalities and
  Constrained Optimization Problems in Function Spaces}, volume~11 of {\em
  MOS-SIAM Series on Optimization}.
\newblock Society for Industrial and Applied Mathematics (SIAM), Philadelphia,
  PA; Mathematical Optimization Society, Philadelphia, PA, 2011.

\bibitem{Wachsmuth2013}
G.~Wachsmuth.
\newblock On {LICQ} and the uniqueness of {L}agrange multipliers.
\newblock {\em Oper. Res. Lett.}, 41(1):78--80, 2013.

\bibitem{Wierzbicki1977}
A.~P. Wierzbicki and S.~Kurcyusz.
\newblock Projection on a cone, penalty functionals and duality theory for
  problems with inequality constraints in {H}ilbert space.
\newblock {\em SIAM J. Control Optimization}, 15(1):25--56, 1977.

\bibitem{Zowe1979}
J.~Zowe and S.~Kurcyusz.
\newblock Regularity and stability for the mathematical programming problem in
  {B}anach spaces.
\newblock {\em Appl. Math. Optim.}, 5(1):49--62, 1979.

\end{thebibliography}

\end{document}